\documentclass[a4paper]{amsart}
\usepackage[utf8]{inputenc}
\title[CR maps into uniformly pseudoconvex hypersurfaces]{Regularity of CR maps into uniformly pseudoconvex hypersurfaces and applications
to proper holomorphic maps}
\author{Josef Greilhuber}
\address{Universität Wien, Oskar-Morgenstern-Platz 1, 1090 Wien, Austria}
\email{josef.greilhuber@univie.ac.at}
\author{Bernhard Lamel}
\address{Universität Wien, Oskar-Morgenstern-Platz 1, 1090 Wien, Austria}
\email{bernhard.lamel@univie.ac.at}
\usepackage{amsmath,amsfonts,amssymb,amsthm}
\usepackage{mathrsfs}
\usepackage{stmaryrd}
\usepackage{svg}
\usepackage{array}
\usepackage{faktor}
\usepackage{graphicx}
\usepackage{caption}
\usepackage{subcaption}
\usepackage{hyperref}
\usepackage{cleveref}
\usepackage{enumerate}
\usepackage{thmtools}
\usepackage{thm-restate}

\newcommand{\C}{\mathbb{C}}
\newcommand{\R}{\mathbb{R}}
\newcommand{\cinfty}{C^{\infty}}

\newcommand{\pp}{{p'}}

\newcommand{\MOne}{M_I^{m,n}}
\newcommand{\DTwo}{D_{II}^m}
\newcommand{\MTwo}{M_{II}^m}
\newcommand{\DThree}{D_{III}^m}
\newcommand{\MThree}{M_{III}^m}
\newcommand{\DFour}{D_{IV}^m}
\newcommand{\MFour}{M_{IV}^m}
\newcommand{\CNp}{\C^{N'}}
\newcommand{\CN}{\C^{N}}
\newcommand{\diffable}[1]{C^{#1}}
\newtheorem{cor}{Corollary}
\newtheorem{lem}{Lemma}
\newtheorem{exa}{Example}
\newtheorem{thm}{Theorem}
\newtheorem{prop}{Proposition}

\theoremstyle{definition}
\newtheorem{defi}{Definition}
\newtheorem{rem}{Remark}

\numberwithin{equation}{section} 

\newif\ifpdfcomments\pdfcommentstrue


\begin{document}
\begin{abstract}
    We study regularity properties 
    of CR maps in positive codimension valued in pseudoconvex manifolds which carry a nontrivial Levi foliation. We introduce an invariant which can be used to deduce that any sufficiently regular CR map
    from a minimal manifold into such a foliated target is either generically smooth or geometrically highly constrained, and to show generic smoothness of sufficiently regular CR transversal CR maps between pseudoconvex hypersurfaces. As an application, we discuss boundary regularity of proper holomorphic maps into bounded symmetric domains.
\end{abstract}
\maketitle

\section{Introduction}
This paper is devoted to the study of  regularity of CR maps into smooth
Levi-degenerate hypersurfaces foliated by complex manifolds and the application of these results to boundary regularity 
of proper holomorphic maps in positive codimension. The positive 
codimensional case is much more challenging than the equidimensional
case in many regards, but also has some specific features which make 
natural answers to regularity problems for mappings a bit different. We refer the 
reader to the discussion in \cite{lamelmir2}, where the authors point out 
some of the salient points, and summarize here those which are important for 
our approach. We consider a CR map $h\colon M \to M'$ with $M \subset \CN$ a CR submanifold, 
and $M'\subset \CNp$ a hypersurface (we shall simply refer to $M$ as the source, and 
$M'$ as the target). All structures and manifolds in this paper are
assumed to be smooth unless explicitly stated otherwise; our technique and our
results extend to other categories as well, as we will outline after stating and 
discussing
our results in the smooth setting first. 

For the purpose of this discussion, the following observations are important. First of all, the typical conclusion of a regularity statement in higher codimension is 
that of {\em generic smoothness} of the map given some a priori regularity, i.e. smoothness on 
a dense, open subset. We cannot drop 
the a priori regularity across a certain
threshhold, as for example Low \cite{Low:1985it} and Stensones \cite{Stensones:1996ek} showed. 
The typical a priori bound that we are going to use are linear in the codimension of the map
considered, which is typical for all known results except for the notable exception of 
mappings between spheres of ``small'' codimension, see e.g. Huang \cite{MR1703603}.

Automatic generic smoothness of all CR maps with such an a priori  regularity
follows if 
the target is of D'Angelo finite type (that is, if it does not contain any formal 
holomorphic curves), as shown in \cite{lamelmir1}. However, the condition that $M'$ is of 
finite type is definitely {\em not necessary} if one excludes certain typical examples of 
non-smooth CR maps. For example, since there always exist a nowhere smooth CR function $\varphi$ (actually, with 
arbitrary finite smoothness prescribed) near a strictly pseudoconvex 
points, if the target contains a complex curve $\gamma (\zeta)$, one obtains a nowhere 
smooth CR map $\gamma\circ \varphi$.

However, as was already observed in \cite{lamelmir1} in the case where the 
target is the tube over the light cone, in many geometrically interesting
situations, this type of behavior is the only exceptional example. Our first main 
theorem describes one such situation. In order to formulate it, we need to
introduce an invariant measuring the (non)degeneracy of a foliation by 
complex manifolds. 

Consider a foliation $\eta$ of $M' \subseteq \C^{N'}$ by complex manifolds $\eta_p$, where
$\eta_p$ denotes the leaf of $\eta$ containing $p$. To each $p\in M'$, we associate
\begin{align*}
    \nu_p :=\max_{0\neq V_p \in T_p\eta} \dim_\C \ker \left( \Bar L_p \rightarrow \mathbb{P}_{T\C^{N'}/T \eta}(\Bar L_p V) \right) - \dim_\C \eta.
\end{align*}
Here $\bar L_p V$ denotes the componentwise derivative
at $p$ of an (arbitrary) smooth extension $V$ of $V_p$, and we project onto the quotient space 
$\faktor{T \C^N}{T\eta} $, where $T\eta \subset T\C^{N'}|_M$ denotes the tangent bundle of $\eta$, i.e. $T_p\eta = T_p\eta_p$.  It turns out that 
this yields a well defined invariant 
$\nu_p$ because as we shall show in \cref{sec:surfaces} the map associating to $\Bar L \in T^{0,1}M'$ and $V \in \Gamma(T\eta)$ the section $\mathbb{P}_{T\C^{N'}/T \eta}(\Bar L V)$ of the quotient bundle $T\C^{N'}/T\eta$ is  tensorial.

\begin{restatable}{thm}{coralways}
\label{cor:always}
Let $M'\subset \C^{N'}$ be a uniformly pseudoconvex hypersurface 
with Levi foliation $\eta$, satisfying $\nu_{p'} = 0 $ for all $p' \in M'$, and let $M$ be a connected \emph{minimal} CR submanifold. Then any $C^{N'-1}$-regular CR map $h: M \rightarrow M'$ is either generically smooth, or it maps $M$ entirely into a single leaf of the foliation, i.e. $h(M) \subseteq \eta_{h(p)}$ for any $p \in M$.
\end{restatable}

If $\nu$ is nonzero, our second main theorem still 
guarantees automatic regularity if the number of positive Levi eigenvalues of the source manifold $M$ is large enough and $h$ is assumed to be CR transversal (which 
is automatically satisfied in many applications).

\begin{restatable}{thm}{cortransversal}
\label{cor:transversal}
Let $M' \subset \C^{N'}$ be a uniformly pseudoconvex hypersurface, and $M \subset \C^N$ a \emph{pseudoconvex} hypersurface with at least $n_+$ positive
Levi eigenvalues. Then any \emph{CR-transversal} CR map $h: M \rightarrow M'$ of regularity 
$C^{N'-n_+}\cap C^2$ is generically smooth near any point $p$ which satisfies $\nu_{h(p)} < n_+$.
\end{restatable}

We note that under a different set of assumptions (which are, as we will see later, not directly related to 
our invariant $\nu$) Xiao \cite[Thm. 1]{regularitymx} obtains everywhere regularity: to be precise,
in his case, he considers maps from strongly pseudoconvex hypersurfaces into 
uniformly pseudoconvex hypersurfaces of the same signature which are $2$-nondegenerate; the point 
is that such maps are automatically
$2$-nondegenerate in the sense of \cite{Lamel:2001vs}, so that one can apply theorems 
from \cite{Lamel:2001kq,Lamel:2004hh}. We discuss the connection with our result later in \cref{sec:2ndeg}, and point out here that $\nu = 0$ implies, in particular, $2$-nondegeneracy 
of the target; our assumptions, however, do not imply that 
that the maps we are considering are $2$-nondegenerate. 
Let us furthermore point out the low codimension results 
in the paper of Kossovskiy, Xiao, and the second author 
\cite{Kossovskiy:2016wx}.

As an application of Theorem~\ref{cor:transversal} we obtain the 
following boundary regularity result for proper holomorphic maps:

\begin{cor}
\label{cor:holomorphic}
Let $\Omega \subseteq \CN$ and $\Omega' \subseteq \CNp$ be domains, and $M \subset \partial \Omega$, $M'\subset \partial \Omega'$ be two hypersurfaces contained in the respective domains' smooth boundary part. Assume that $\Omega'$ is uniformly pseudoconvex at $M'$, and $M$ is pseudoconvex with at least $n_+$ positive Levi eigenvalues. Then every holomorphic map $H: \Omega \rightarrow \Omega'$ which extends as a $C^{N'-n_+}\cap C^2$-regular map to $M$ and maps $M$ into $M'$, is generically smooth (on $\bar \Omega$) near any $p \in M$ satisfying $\nu_{H(p)} < n_+$.
\end{cor}

 In \Cref{sec:domains} we discuss in detail  proper holomorphic maps into pseudoconvex domains as an 
application of \Cref{cor:transversal}, and prove 
a theorem which deals with maps into boundaries of classical symmetric domains. By a careful study of the geometry of the smooth boundary components, we can calculate 
the invariant $\nu$ in each case, and obtain the following theorem, which significantly extends the results obtained by Xiao in \cite{regularitymx}: the source manifolds we can consider are not required to be strictly pseudoconvex any longer. The price we have to pay is that  we have 
to assume higher a priori regularity and only 
obtain generic smoothness. 

\begin{thm}
\label{thm:zusammenfassung}
Let $M \subset \C^N$ be a $\cinfty$-smooth pseudoconvex hypersurface, and assume that its Levi form has exactly $n_+$ positive eigenvalues everywhere. Denote by $M'$ the smooth part of the boundary of a classical symmetric domain $\Omega \subseteq \C^{N'}$. Then every CR-transversal CR map $h:M \rightarrow M'$ of regularity $C^{N'-n_+}$ is smooth on a dense open subset of $M$, given that
\begin{enumerate}[\rm (1)]
    \item $\Omega = D_I^{m,n}$ for $m,n\geq 2$ and $n_+ \subseteq \{m+n-3,m+n-2\}$,
    \item $\Omega = D_{II}^{m}$ for $m \geq 4$ and $n_+ \subseteq \{2m-7,\dots,2m-4\}$,
    \item $\Omega = D_{III}^{m}$ for $m \geq 2$ and $n_+ = m-1$ or
    \item $\Omega = D_{IV}^{m}$ for $m \geq 2$, $M$ is minimal and $n_+ \leq m-2$.
\end{enumerate}
\end{thm}

Let us note that \Cref{thm:zusammenfassung} in particular applies to the setting
of (appropriate) {\em boundary values of proper mappings between classical symmetric domains}.

The conditions on the number of positive Levi eigenvalues given in \Cref{thm:zusammenfassung} are sharp in the following sense: In the case of $D_I^{m,n}$, $D_{II}^m$ and $D_{III}^m$, there are pseudoconvex hypersurfaces $M$ satisfying $n_+ = m+n-4$, $n_+ = 2m-8$ and $n_+ = m-2$, respectively, such that there exist nowhere smooth, but arbitrarily often continuously differentiable CR-transversal CR embeddings $h:M \rightarrow M'$.

On the other hand, there exists no CR transversal map $h:M \rightarrow M'$ at all if the number of positive Levi eigenvalues of $M$ exceeds the upper limit given in \Cref{thm:zusammenfassung}, which is just the number of positive Levi eigenvalues of $M'$.

\begin{rem}
We first remark that one can obtain
results in the real-analytic category with exactly the same assumptions. The conclusion in this setting is that the map $h$ extends to a holomorphic map in a full neighbourhood of an open, dense subset 
of the source manifold. For this, 
one uses the result of Mir \cite{Mir:2017dg} on real-analytic regularity. We also 
remark that if both source and target 
are real-algebraic, one can use the 
proof of the algebraicity result of 
Coupet, Meylan, and Sukhov \cite{MR1666972}
to conclude that the map $h$ is real-algebraic (this conclusion is global
in nature). 
\end{rem}

\begin{rem}
An interesting observation in the algebraic case is that  \Cref{thm:zusammenfassung} and \Cref{cor:holomorphic} yield a complete list of pairs of classical symmetric domains $(\Omega_1,\Omega_2)$ such that every proper holomorphic map $H: \Omega_1 \rightarrow \Omega_2$, which extends to $\partial \Omega_1$ with sufficient initial regularity, and does not map $\partial \Omega_1$ entirely into the non-smooth part of $\partial \Omega_2$, is necessarily algebraic.
\end{rem}

\begin{rem}
We finally note that all of our results above apply as well in the case where the source 
manifold is not embedded, but rather an ``abstract'' CR structure with the microlocal extension property; in that case, 
we have to apply the results of \cite{MR4060572} instead of \cite{lamelmir1}. We tried to avoid this more technical aspect in the presentation here, the reader 
is invited to make the obvious changes to the formulations if needed.
\end{rem}

This paper developed from the first author's master's thesis \cite{thesis}, where (slightly weaker) versions of \Cref{cor:always}, \Cref{cor:transversal} and \Cref{thm:zusammenfassung} are first proven.



\section{Preliminaries}

\subsection{CR manifolds and the Levi foliation} In this section, we will recall some basic notions and fix notation. We will be considering smooth CR submanifolds of complex Euclidean space, which we will denote by $M \subset \C^{N}$ or $M' \subset \C^{N'}$, respectively, and CR maps $h: M \rightarrow M'$. We will write $T^{0,1} M = \C TM \cap T^{0,1} \C^N$, and denote 
by $J$ the standard complex structure operator. 

A continuously differentiable map $h:M \rightarrow M'$ is called a CR map if it preserves the CR structure of its domain, i.e. if $h_\ast T_q^{0,1} M \subseteq T_{h(q)}^{0,1} M'$ for all $q\in M$. If we denote by $\iota$ the embedding of $M'$ into $\C^{N'}$, an equivalent characterization is that $\iota \circ h: M \rightarrow \C^{N'}$ is a CR map, which just means that each coordinate component of $h$ is a CR function.

We recall that the Levi form of a hypersurface $M\subset \C^{N'}$ is 
defined by $\mathcal{L} (\bar L, \bar \Gamma) = \frac{1}{2i}[\bar L, \Gamma] \mod T^{0,1} M \oplus T^{1,0} M $; given a defining function $\rho$ for 
$M$, we define $\Theta = i (\partial \rho - \bar \partial \rho)$, and refer to 
$\mathcal{L}_{\Theta} (\bar L, \bar \Gamma) = \frac{1}{2i}\Theta ( [\bar L, \Gamma] ) $
as a scalar Levi form for $M$. 

Our target manifolds will be \emph{uniformly pseudoconvex hypersurfaces}, i.e. real hypersurfaces of $\C^{N'}$ with positive semidefinite Levi form, and a constant number of zero and positive eigenvalues everywhere, respectively. It will turn out that these are foliated by complex manifolds.

In this paper, a \emph{foliation} $\eta$ of an $n$-dimensional (real) manifold $M$ is a collection $\{\eta_q: q\in M\}$ of $k$-dimensional immersed submanifolds, where $q \in \eta_q$, which partitions $M$, i.e. $\eta_p$ and $\eta_q$ are either disjoint or identical for any two $p,q\in M$, and such that for any $p \in M$, there exists a neighborhood $O$ of $p$ and coordinates $\phi:O \rightarrow \R^n$ such that for any $q \in O$, \emph{the connected component} of $\eta_q \cap O$ containing $q$ is just given by the coordinate plane $(\phi_1(p),\dots,\phi_{n-k}(p),\cdot,\dots,\cdot) \cap \phi(O)$.

The bundle $T\eta := \bigcup_{q \in M} T_q\eta_q$ of tangent spaces to leaves then forms a smooth integrable distribution on $TM$. We will also consider the bundle $T^{0,1}\eta := \bigcup_{q \in M} T^{0,1}_q\eta_q$ of CR tangent spaces to leaves, and always write $T_q\eta := T_q\eta_q$ and $T_q^{0,1}\eta := T_q^{0,1}\eta_q$ for simplicity.

If the rank of the Levi form of a CR manifold $M$ is constant in a neighborhood $U$ of a point $p \in M$, there exists a foliation $\eta$ of $U$ by complex manifolds, such that the Levi null space at any $q \in U$ is precisely given by the CR tangent space at $q$ to the the leaf of the foliation through $q$, henceforth denoted by $T^{0,1}_q\eta$. This foliation, discovered in the hypersurface case by Sommer \cite{sommer} and proven to exist in general CR submanifolds by Freeman \cite{freeman1} is thus called the \emph{Levi foliation}.

\begin{thm}
\label{thm:levifoliation}
Let $M \subset \C^N$ be a CR manifold, and suppose that its Levi form has constant rank. Then there is a foliation $\eta$ of $U$ by complex manifolds, such that the Levi null spaces $\mathcal{N}_q \subseteq T_q^{0,1}M$ for $q \in U$ are given by $T_q^{0,1}\eta$.
\end{thm}

\begin{proof}
Let $N_q = \{\frac{1}{2}(\Bar L_q + L_q), \Bar L_q \in \mathcal{N}_q\}$. Because the rank of the Levi null space is constant across $M$, the union $N = \bigcup_{q \in M} N_q$ yields a smooth real distribution on $M$. By Frobenius' theorem, integrability of this distribution is equivalent to the submodule $\Gamma_q(N)$ of germs of sections of $N$ at $q$ being closed under taking Lie brackets, for every $q \in M$.
Since for any $\Bar L, \Bar \Gamma \in \Gamma(T^{0,1}M)$ we have
\begin{align*}
    [\tfrac{1}{2}(\Bar L + L), \tfrac{1}{2}(\Bar \Gamma + \Gamma)] = \tfrac{1}{4}[\Bar L, \Bar \Gamma] + \tfrac{1}{4}[L, \Gamma] + \tfrac{1}{4}[L, \Bar \Gamma] + \tfrac{1}{4}[\Bar L, \Gamma] = -\Im\left(\mathcal{L}(\Bar L, \Bar \Gamma)\right),
\end{align*}
a given germ of a vector field $\frac{1}{2}(\bar L+L) \in \Gamma_q(T^cM)$ is a section of $N$ if and only if $[\frac{1}{2}(\bar L+L),\Gamma_q(T^cM)] \subseteq \Gamma_q(T^cM)$. Taking two sections $V, W \in \Gamma_q(N)$, we see thus that $[V,W] \subseteq [\Gamma_q(N),\Gamma_q(T^cM)] \subseteq \Gamma_q(T^cM)$, and by the Jacobi identity, 
\begin{align*}
    \left[[V,W],\Gamma_q(T^cM)\right] &\subseteq \left[V,[W,\Gamma_q(T^cM)]\right] - \left[W,[V,\Gamma_q(T^cM)]\right]  \\ &\subseteq 
    \left[V,\Gamma_q(T^cM)\right] - \left[W,\Gamma_q(T^cM)\right] \subseteq \Gamma_q(T^cM),
\end{align*}
showing that $\Gamma_q(N)$ is indeed closed under taking Lie brackets. Therefore, the Levi foliation exists, and since $N_q \subset T_q\C^N$ is a complex subspace for any $q \in M$, the leaves of this foliation are complex manifolds.
\end{proof}

\subsection{Irregular CR maps and formal holomorphic foliations}
Even though we are interested in the regularity of mappings, 
our results are obtained in a contrapositive way: We show
that the existence of irregular maps forces some geometric
property (namely, the existence of complex
varieties, see \Cref{thm:main} below). As a guiding principle, we therefore review a couple of natural instances in 
which irregular maps exist. We begin by considering CR functions, in a slight adaptation of \cite[Theorem 2.7]{bx}.

\begin{exa}
Let $M \subset \C^N$ be a strongly pseudoconvex CR hypersurface and $p \in M$. Then there exists a neighborhood $O \subseteq \C^N$ of $p$ such that for each $k \in \mathbb{N}_{\geq 1}$ there is a $C^k$-smooth CR function $\phi:O \cap M \rightarrow \C$ which is nowhere smooth on $O \cap M$.
\end{exa}

As an immediate conseqence, there exist nowhere smooth CR maps from $M$ into $M'$ if the target manifold $M'$ contains a complex curve $\Gamma$. Indeed, any parametrization $t \mapsto \gamma(t)$ of $\Gamma$ is a smooth CR immersion of $\C$ into $M'$, hence $\gamma \circ \phi: M \rightarrow M'$ provides a nowhere smooth CR function of regularity $C^k$.
We obtain another, more general set of examples from targets of the form $M' = \hat M \times \C \subset \C^{N+1}$ and CR functions $\hat h: M \rightarrow \hat M$. Here, the map $(\hat h,\phi): M \rightarrow \hat M \times \C$ is a CR map, since each of its components is a CR map, and it is nowhere smooth because $\phi$ is. In \cite{lamelmir1}, Lamel and Mir prove a result in the other direction, essentially stating that near a generic point, any nowhere smooth CR map formally exhibits the structure of these latter examples.

\subsection{The formal foliation theorem}\label{ss:formalfol}

Before we state the main technical theorem that we 
are going to use, we introduce some necessary concepts. A \emph{formal holomorphic submanifold} $\Gamma$ of dimension $r$ at a point $p \in \C^{N'}$ is simply a formal power series $\Gamma \in \C\llbracket t_1,\dots,t_r\rrbracket^{N'}$, $\Gamma = \sum_{\alpha \in \mathbb{N}^r}\gamma_\alpha t^\alpha$ satisfying $\gamma_0 = p$ and $\mathrm{rk} \left(\Gamma'(0)\right) = r$. It is \emph{tangential to infinite order} to a set $S \subseteq \C^{N'}$ if for any germ of a $\cinfty$-smooth function $\rho$ vanishing on $S$, the composition of $\Gamma$ with the Taylor series of $\rho$ at $p$ vanishes to infinite order. If $M$ is a CR manifold and $(\Gamma_q)_{q \in M}$ is a family of such formal holomorphic submanifolds, we call this family a \emph{CR family} if each of its coefficients is a CR map $M \rightarrow \C^{N'}$.

It turns out that the structural property of the target which forces smoothness of CR maps is the number of different directions into which successive CR derivatives of gradients of defining functions can point. This motivates the introduction of the following numerical invariants. For a CR map $h: M \rightarrow \C^{N'}$, let
\begin{align*}
    r_0(p) &:= \dim_\C \left< \left\{\rho_w\circ h (p): \rho \in \mathscr{I}_{h(M)}(h(p))\right\}\right>, \\
    r_k(p) &:= \dim_\C \left< \left\{\Bar L_1\dots \Bar L_j (\rho_w \circ h)(p): \rho \in \mathscr{I}_{h(M)}(h(p)), \Bar L_1,\dots,\Bar L_j \in \mathcal{V}_p(M), 0\leq j \leq k\right\}\right>,
\end{align*}
where we write $\mathcal{V}_p(M)$ for the set of germs of CR vector fields at $p$, and $\mathscr{I}_S(p)$ for the ideal of germs of smooth functions at $p$ which vanish on a given set $S$. The \emph{complex gradient} $\rho_w = \left(\frac{\partial \rho}{\partial w_1},\dots,\frac{\partial \rho}{\partial w_{N'}}\right)$ is considered here as a vector in $\C^{N'}$. The function $q \mapsto r_k(q)$ is integer valued and lower semicontinuous as it is given by the rank of a collection of continuously varying vectors. Of course, $r_k(p)$ is only defined if $h \in C^k$, since $\rho_w\circ h$ is only as regular as $h$ is. To extract a global invariant of $h$, let $r_k$ be the maximum value such that $r_k(p) \geq r_k$ on a dense open subset of $M$. We are now in a position to state the formal foliation theorem of Lamel and Mir (Theorem 2.2 in \cite{lamelmir1}).

\begin{thm}
\label{thm:main}
Let $M \subset \C^{N}$ be a $\cinfty$-smooth minimal CR submanifold, $k,l \in \mathbb{N}$ with $0 \leq k \leq l \leq N'$ and $N'-l+k \geq 1$ be given integers and $h:M \rightarrow \C^{N'}$ be a CR map of class $C^{N'-l+k}$. Assume that $r_k \geq l$ and that there exists a non-empty open subset $M_1$ of $M$ where $h$ is nowhere $\cinfty$. Then there exists a dense open subset $M_2 \subseteq M_1$ such that for every $p\in M_2$, there exists a neighborhood $V\subseteq M_2$ of $p$, an integer $r\geq 1$ and a $C^1$-smooth CR family of formal complex submanifolds $(\Gamma_\xi)_{\xi \in V}$ of dimension $r$ through $h(V)$ for which $\Gamma_\xi$ is tangential to infinite order to $h(M)$ at $h(\xi)$, for every $\xi \in V$.
\end{thm}

The rank $r$ of the family of holomorphic manifolds in the statement of this theorem merely serves as a reminder that in concrete cases, one can hope for a rank of more than one. Since there is no condition given when this might occur, for black-box applications of this theorem we will have to be satisfied with CR families of holomorphic curves with nonvanishing derivative, which can always be obtained by simply restricting $\Gamma_q = \sum_{\alpha \in \mathbb{N}^r} \gamma_\alpha(q) t^\alpha$ to $t = (t_1,0,\dots,0)$.

Let us remark that if $h$ is not $\cinfty$-smooth on a dense open subset of $M$, there exists an open subset $O\subseteq M$ such that $h$ is nowhere $\cinfty$-smooth on $O$. The reason is simply that the set of all points $p \in M$ such that $h$ is $\cinfty$-smooth on a neighborhood of $p$ is open. If this set is not dense, then the complement of its closure is a non-empty open subset of $M$, where, by definition, $h$ is nowhere $\cinfty$-smooth.

Another interesting point to note is that while the formal complex manifolds obtained from \Cref{thm:main} are tangential to infinite order to the image $h(M)$, infinite tangency to a non-smooth set is not nearly as strong as one might think at first sight.
As a toy example, take a nowhere smooth, but $C^1$ function $\phi: \R \rightarrow \R$ and consider its graph $S := \{(x,\phi(x)), x \in \R \} \subset \R^2$. Then any function $\rho \in \cinfty(\R^2)$ vanishing on $S$ must already vanish to infinite order there by the following argument: If either $\rho_x$ or $\rho_y$ did not vanish at a point $(x,\phi(x))$, the implicit function theorem would yield a smooth parametrization of $S$ near that point, which does not exist. Thus both $\rho_x$ and $\rho_y$ vanish on $S$, and the argument may proceed at infinitum. The $y$-Axis is therefore tangential to infinite order to $S$ in the sense of \Cref{thm:main}, while not even being tangential to first order in the usual sense.
However, if $h(M) \subseteq M'$ for some smooth manifold $M'$, then tangency to infinite order to $h(M)$ clearly implies tangency to infinite order to $M'$.

To apply \Cref{thm:main}, we need $0 \leq k \leq l \leq N'$ such that $r_k\geq l$. It is always possible to choose $k=l=0$, but if $h$ maps $M$ into a CR submanifold $M' \subseteq \C^{N'}$, a slight improvement holds (Lemma 6.1 in \cite{lamelmir1}).

\begin{lem}
\label{lem:r0always}
Let $M\subset \C^N$ be a $\cinfty$-smooth CR submanifold and $h: M \rightarrow \C^{N'}$ be a continuous CR map. If there exists a $\cinfty$-smooth CR submanifold $M'\subset \C^{N'}$ such that $h(M) \subseteq M'$, then $r_0\geq N'-n'$, where $n' = \dim_{CR}M'$. In particular, if $M'$ is maximally real, then $r_0 = N'$.
\end{lem}

If it is guaranteed that enough CR directions tangential to $h(M)$ exist along which $M'$ behaves like a Levi nondegenerate manifold, we can say more about the first derivatives of gradients, yielding a bound on $r_1$. We record here for later use a result similar to Lemma 6.2. in \cite{lamelmir1}.

\begin{lem}
\label{lem:r1}
Consider a $C^\infty$-smooth CR submanifold $M\subset \C^N$, a $C^\infty$-smooth real hypersurface $M' \subset \C^{N'}$ and a continuously differentiable CR map $h: M \rightarrow M'$ mapping $p \in M$ to $\pp \in M'$. If $h$ is immersive at $p$ and a scalar Levi form $\mathcal{L}_\Theta$ of $M'$ restricts to a nondegenerate Hermitian form on $h_\ast T_p^{0,1}M$, then $r_1 \geq \dim_{CR}M + 1$ on a neighborhood of $p$.
\end{lem}

\begin{proof}
Since we are in a purely local setting, we may assume that $\mathcal{L}_\Theta$ arises from a defining function $\rho$ of $M'$, such that for any two CR vectors $\Bar \Gamma = \sum_{j=1}^{N'} \Bar \Gamma_j \frac{\partial}{\partial \Bar w_j}|_\pp$ and $\Bar L = \sum_{k=1}^{N'} \Bar L_k \frac{\partial}{\partial \Bar w_k}|_\pp$ we have
\begin{align*}
    \mathcal{L}_\Theta(\Bar \Gamma, \Bar L) = \sum_{j,k=1}^{N'} \frac{\partial^2 \rho}{\partial w_j \partial \Bar w_k}( \pp) \Gamma_j \Bar L_k.
\end{align*}
By definition $\Bar L \rho_w = \sum_{j=1}^{N'} \Bar L_k \frac{\partial^2 \rho}{\partial w_j \partial \Bar w_k}(\pp)$, so using the standard scalar product on $\C^{N'}$ we can express $\mathcal{L}_\Theta(\Bar \Gamma, \Bar L) = \left((\Bar \Gamma_1,\dots,\Bar \Gamma_{N'})|\Bar L \rho_w \right)_{\C^{N'}}$. Nondegeneracy of the restricted Levi form on $h_\ast T_p^{0,1}M$ precisely means that the map $h_\ast \Bar L \mapsto \mathcal{L}_\Theta(\cdot, h_\ast \Bar L)$ is an isomorphism of $h_\ast T_p^{0,1}M$ and the space of antilinear functionals on $h_\ast T_p^{0,1}M$. Since $h$ is immersive, $h_\ast$ is an isomorphism between $T_p^{0,1}M$ and $h_\ast T_p^{0,1}M$. The map associating to each $\Bar L \in T_p^{0,1}M$ the antilinear functional $\mathcal{L}_\Theta(\cdot, h_\ast \Bar L) = \left( \hspace{1pt} \cdot \hspace{2pt} | \Bar L (\rho_w \circ h)\right)_{\C^{N'}}$ is thus an isomorphism, in particular implying that $\dim_\C \{\Bar L (\rho_w \circ h): \Bar L \in T_p^{0,1}M\} = \dim_{CR} M$.
Furthermore, the complex gradient $\rho_w(\pp)$ itself is linearly independent of $\Bar L (\rho_w \circ h)$ for any nonzero $\Bar L \in T_\pp^{0,1}M'$ by the following argument. For any $\Bar \Gamma = \sum_{j=1}^{N'} \Bar \Gamma_j \frac{\partial}{\partial \Bar w_j}|_\pp \in T_\pp^{0,1}M'$, tangency implies that
\begin{align*}
    \Gamma \rho &= \sum_{j=1}^{N'} \Gamma_j \frac{\partial \rho}{\partial w_j}(\pp) = \left( (\Bar \Gamma_1,\dots,\Bar \Gamma_{N'})| \rho_w(\pp) \right)_{\C^{N'}} = 0.
\end{align*}
Thus $\rho_w(\pp)$ lies in the orthogonal complement of $\left\{(\Bar \Gamma_1,\dots,\Bar \Gamma_{N'}): \Bar \Gamma \in T_\pp^{0,1}M'\right\}$ while $\Bar L (\rho_w \circ h)$ does not, showing linear independence. This implies $r_1(p) \geq \dim_{CR} M + 1$ and since $r_1$ is lower semicontinuous and integer valued, $r_1 \geq \dim_{CR} M + 1$ holds on a neighborhood of $p$ as claimed.
\end{proof}

As we will have to treat non-immersive maps as well, let us note the following simple, but slightly clunky consequence of the previous proof.

\begin{cor}
\label{cor:r1}
If for a $C^\infty$-smooth CR submanifold $M\subset \C^N$, a $C^\infty$-smooth real hypersurface $M' \subset \C^{N'}$ and a continuously differentiable CR map $h: M \rightarrow M'$ mapping $p \in M$ to $\pp \in M'$ there exists a CR submanifold $S \subset M$ containing $p$, such that the restricted map $h|_S$ satisfies the hypothesis of \Cref{lem:r1}, then $r_1 \geq \dim_{CR} S + 1$ on a neighborhood of $p$ in $M$.
\end{cor}

\begin{proof}
Take a basis $(\bar L_j)_{j=1}^{\dim_{CR}S}$ of $T^{0,1}_p S$.
Retracing the proof of \Cref{lem:r1}, we see that for any defining function $\rho$ of $M$, the vectors $\rho_w(p')$ and $\bar L_j(\rho_w\circ h), 1 \leq j \leq \dim_{CR} S$ are linearly independent, hence $r_1(p) \geq \dim_{CR} S + 1$. But $r_1(q)$ is lower semicontinuous in $q$ on $M$, hence $r_1 \geq \dim_{CR} S + 1$ in a neighborhood of $p$ as claimed.
\end{proof}

\section{The invariant $\nu$ and the proof of \cref{cor:always}}

As an example of a hypersurface foliated by complex manifolds, where an unconditional regularity result must necessarily fail, Lamel and Mir consider the \emph{tube over the light cone} $M' := \{(z_1,\dots,z_{N'-1},z_{N'}): \Re(z_1)^2 + \dots + \Re(z_{N'-1})^2 = \Re(z_{N'})^2, \ z_{N'} \neq 0\}$. They obtain the following result (Corollary 2.6 in \cite{lamelmir1}).

\begin{thm}\label{thm:lightcone}
Let $M \subset \C^N$ be a $\cinfty$-smooth minimal CR submanifold and $M' \subseteq \C^{N'}$ be the tube over the light cone. Then every CR map $h:M \rightarrow M'$, of class $C^{N'-1}$ and of rank $\geq 3$, is $\cinfty$-smooth on a dense open subset of $M$.
\end{thm}

The proof given in \cite{lamelmir1} and \cite{lamelmir2} makes quite ingenious use of the simple structure of $M'$, and is thus not easily adaptable to more general settings. In this section, we shall carefully define the invariant $\nu$ mentioned in the introduction \eqref{def:nu}, and 
show how it can be used to generalize the observation of \cref{thm:lightcone} to the more general situation of pseudoconvex hypersurfaces whose Levi form is of constant rank. We will later see that
this class of examples covers not only the tube over the light cone, but also the smooth part of the boundary of all classical irreducible symmetric domains. Mappings into such targets will be discussed in \cref{sec:csd}.

\subsection{Maps into uniformly pseudoconvex hypersurfaces}
\label{sec:surfaces}
In view of the Levi foliation (\Cref{thm:levifoliation}), \Cref{thm:main} might allow for nowhere smooth maps into a uniformly pseudoconvex hypersurface, since there at least exist complex manifolds tangential to infinite order to the target manifold, contrary to the simpler case of manifolds of D'Angelo finite type. Indeed, any formal complex manifold tangential to infinite order to $M'$ is necessarily tangential to the Levi foliation.

\begin{lem}
\label{lem:fiberdirection}
Let $M' \subset \C^{N'}$ be a uniformly pseudoconvex hypersurface with its Levi foliation $\eta$, and let $p' \in M'$. Suppose there exists a formal complex curve $\Gamma = \pp + t \gamma_t + \frac{t^2}{2} \gamma_{tt} + \dots$ tangential to second order to $M'$ at $\pp$. Then $\gamma_t \in T_\pp\eta$.
\end{lem}

\begin{proof}A formal complex curve $\Gamma = \pp + t\gamma_t + \frac{t^2}{2}\gamma_{tt} + \dots$ is tangential to second order to $M'$ if and only if the curve $\tilde \gamma(t) = \pp + t\gamma_t + \frac{t^2}{2}\gamma_{tt}$ arising from the truncated power series is. Choosing a positive semidefinite scalar Levi form $\mathcal{L}_\Theta$ arising from a defining function $\rho$, we have that $\mathcal{L}_\Theta(\frac{1}{2}(\gamma_t + iJ\gamma_t), \frac{1}{2}(\gamma_t + iJ\gamma_t)) = \frac{d^2}{dt d\bar t}|_{t=0} \rho \circ \tilde\gamma = 0$, since $\rho \circ \tilde \gamma$ vanishes to second order. But by \Cref{thm:levifoliation}, the null space of $\mathcal{L}_\Theta|_\pp$ is given by $T^{0,1}_\pp\eta$, implying that $\gamma_t \in T_\pp \eta$.
\end{proof}

Our main technical tool will be a tensorial quantity measuring obstructions to the existence of CR sections of $T\eta$. We denote by $T^\bot\eta := \bigcup_{p\in M'} (T_p\eta_p)^\bot$ the bundle of orthogonal complements in $T\C^{N'}$ of tangent spaces to leaves.

\begin{lem}
\label{lem:tensor}
Let $M'$ be a manifold endowed with a foliation $\eta$ by complex manifolds. There exists a tensor field $R \in \mathcal{V}(M')^\ast \otimes \Gamma(T\eta)^\ast \otimes \Gamma(T^\bot\eta)$ such that for every $\Bar L \in \mathcal{V}(M')$ and $\psi \in \Gamma(T\eta)$, we have $\mathbb{P}_{T^\bot_p\eta_p}(\Bar L|_{p} \psi ) = R_p(\Bar L |_p, \psi |_p)$. For any $V_p \in T_p\eta_p$, the kernel of $R_p(\cdot,V_p)$ contains $T^{0,1}_p\eta_p$.
\end{lem}

\begin{proof}
Define $R(\Bar L, \psi) = \mathbb{P}_{T^\bot\eta}(\Bar L \psi )$. Evidently, $R$ is $C^1$-linear in the first slot, as directional derivatives always are. For two sections $V$ and $W$ of $T\eta$ and $f \in C^1(M)$, we have that $\Bar L|_{p} (V + fW) = \Bar L|_{p} V + f \Bar L|_{p} W + (\Bar L|_{p} f) W$. The last term is canceled by the projection onto $T^\bot_p\eta_p$, thus $R$ is also $C^1$-linear in the second slot, hence $R$ is a tensor.

Consider now $V_p \in T_p\eta_p$. We may construct a section $V \in \Gamma(T\eta)$ satisfying $V|_p = V_p$, which is holomorphic on $\eta_p$ and smooth on $M'$. First, we choose a holomorphic parametrization $\phi$ for $\eta_p$, extend $\phi^{-1}_\ast V_p$ to a constant vector field $\tilde V$ and note that $\phi_\ast \tilde V$ is holomorphic, since $D\phi$ has holomorphic components and $\tilde V$ is constant. To obtain a vector field, we then simply extend the result smoothly to $M'$.
But now, $\Bar L|_{p} V = 0$ if $\Bar L \in \Gamma(T^{0,1}\eta)$, since $V$ is holomorphic on $\eta_p$ and $\Bar L|_p$ only takes derivatives along $\eta$. Therefore, $T_p^{0,1}\eta \subseteq \ker R_p(\cdot,V_p)$.
\end{proof}

Next we need to deal with the issue that the CR map $h$ of interest is neither assumed to be immersive nor $C^\infty$-smooth. After carefully verifying that nothing goes wrong, we will obtain the following result on obstructions to the existence of nowhere smooth CR maps.

\begin{prop}
\label{prop:pointwise}
Consider a uniformly pseudoconvex hypersurface $M'$ with its Levi foliation $\eta$, a CR manifold $M$ and a $C^1$-smooth CR map $h: M \rightarrow M'$ mapping a point $p\in M$ to $p' \in M'$. Suppose there exists a $C^1$-smooth CR family of formal complex curves $\left(\Gamma_q\right)_{q \in O}$ defined on a neighborhood $O \subseteq M$ of $p$ such that $\Gamma_q$ is tangential to second order to $M'$ at $h(q)$ for each $q\in O$. Then $\gamma_t(p) \in T_\pp\eta_\pp$, and $h_\ast T^{0,1}_pM \subseteq \ker R_\pp(\cdot,\gamma_t(p))$.
\end{prop}

\begin{proof}
By Lemma \ref{lem:fiberdirection}, we know that at each point $q \in O$, $\gamma_t(q) \in T_{h(q)}\eta_{h(q)}$, since $\Gamma_q$ is a formal complex curve tangential to second order to $M'$ at $h(q)$.

Consider now $\Bar L \in \mathcal{V}(M)$ such that $h_\ast \Bar L|_p \neq 0$. Choosing a two-dimensional real submanifold $S \subseteq O$ such that $\Bar L|_p$ is tangential to $S$, the derivative of $h|_S$ has full rank at $p$, and hence $h|_S$ is a local embedding around $p$. We may thus extend $\gamma_t \circ h^{-1}|_{h(S)}$, defined on $h(S)$, to a section $\tilde \gamma_t \in \Gamma(T\eta)$ defined on an open neighbourhood of $\pp$.
Since $\gamma_t$ and $\tilde \gamma_t \circ h$ agree on $S$ and $\Bar L_p$ only takes derivatives along $S$, it follows that
\begin{align*}
    R_\pp(h_\ast \Bar L|_p, \gamma_t(p)) = \mathbb{P}_{T_\pp^\bot \eta_\pp}\left( h_\ast \Bar L|_p \tilde \gamma_t \right) =
    \mathbb{P}_{T_\pp^\bot \eta_\pp}\left( \Bar L|_p (\tilde \gamma_t \circ h) \right) =
    \mathbb{P}_{T_\pp^\bot \eta_\pp}\left( \Bar L|_p \gamma_t \right) = 0,
\end{align*}
implying that $h_\ast \Bar L|_p \in \ker R_\pp(\cdot,\gamma_t(p))$.
\end{proof}

In order to apply \Cref{thm:main} to our situation, we are going to use the numerical quantity $\nu$ already mentioned in the introduction, which measures the size of $\ker R(\cdot, V)$, as well as a method of computing it.

\begin{lem}
\label{def:nu}
Let $M'$ be a uniformly pseudoconvex hypersurface with Levi foliation $\eta$. For $p' \in M'$, we define 
\[\nu_\pp = \max_{0 \neq V \in T_\pp \eta} \dim_\C \ker R_\pp(\cdot, V) - \dim_\C \eta.\]
Then $\nu$ is a biholomorphic invariant, and 
the function $q \rightarrow \nu_q$ is nonnegative, integer valued and upper semicontinous. 
\end{lem}

\begin{proof} For biholomorphism invariance, let $H:\C^{N'} \rightarrow \C^{N'}$ be a local biholomorphism near $p'$, and consider $H(M')$. Its Levi foliation is $H(\eta)$. For all $\Bar L \in \mathcal V(M')$ and $V \in \Gamma(T\eta)$, we have $(H_\ast \Bar L)(H_\ast V) = H_\ast (\Bar L V)$ as the Jacobian of $H$ is holomorphic again and thus commutes with taking antiholomorphic derivatives. Furthermore, projecting $H_\ast(T\eta^\bot)$ onto $TH(\eta)^\bot$ is an isomorphism as $H_\ast(T\eta^\bot)$ is another complement to $H_\ast(T\eta) = TH(\eta)$. Therefore $R$ transforms under biholomorphism via pre- and postcomposition with linear isomorphisms, so $\nu$ is invariant. Here we chose $T\eta^\bot$ instead of the more natural, isomorphic bundle $T \C^{N'}/ T\eta$ because it simplifies later calculations.

Clearly $\nu_q$ is integer-valued, and since $T_q^{0,1}\eta \subseteq \ker R(\cdot, V)$, it is nonnegative. Let $\dim_\C \eta =: K$. To show upper semicontinuity, we need to show that for all $\pp$ in $M'$ and $l \in \mathbb{N}$, $\nu_\pp \leq l$ implies $\nu_q \leq l$ for all $q$ in a neighborhood of $\pp$. We observe first that for $l \in \mathbb{N}$, $\dim_\C \ker R_\pp(\cdot, V) \leq K + l$ if and only if $\mathrm{rk} R(\cdot, V) \geq N'-1-K-l$. This condition is equivalent to some $(N'{-}1{-}K{-}l)$-minor of the matrix representation of $R(\cdot, V)$ with respect to a choice of smooth local frames of $T^{0,1}M'$ and $T^\bot\eta$ being nonzero. By homogeneity, we have $\nu_\pp \leq l$ if and only if for each $V \in \mathbb{S}^{2K-1} \subseteq T_\pp$, a (possibly different) such minor of $R(\cdot,V)$ does not vanish. Hence, $\nu_\pp \leq K+l$ if and only if the square sum $\sum_{j}|m_j|^2$ over all such minors does not vanish on $\{\pp\}{\times}\mathbb{S}^{2K-1}$. By compactness of the sphere, there is a neighborhood $O$ of $\pp$ such that $\sum_{j}|m_j|^2$ does not vanish on $O{\times}\mathbb{S}^{2K-1}$, showing that $\nu_q \leq l$ on $O$, which implies upper semicontinuity.
\end{proof}

To calculate $\nu_\pp$, the following setup will be helpful.
\begin{lem}
\label{lem:setup}
Let $M'$ be a pseudoconvex hypersurface foliated by complex manifolds of dimension $K$. Consider a point $\pp \in M'$ and an $(N'{-}K)$-dimensional complex manifold $\Sigma$ through $\pp$ such that $T_\pp\eta \oplus T_\pp \Sigma = T_\pp \C^{N}$. If $S:=M' \cap \Sigma$ is strongly pseudoconvex, then the Levi form of $M'$ has exactly $K$ zero eigenvalues on an open neighborhood of $\pp$, making $M'$ a uniformly pseudoconvex hypersurface and $\eta$ its Levi foliation.

Furthermore, as in Lemma \ref{lem:tensor}, the map $R^S \in \mathcal{V}(S)^\ast\otimes \Gamma(T\eta|_S)^\ast \otimes \Gamma(T^{\bot}\eta|_S)$ given by $R^S(\Bar L,V) = \mathbb{P}_{T^\bot\eta}(\Bar L V)$ is a tensor, and $\max_{0 \neq V \in T_\pp \eta} \dim_{\C} \ker R_\pp^S(\cdot , V) = \nu_\pp$.
\end{lem}

\begin{proof}
Let $\Theta$ be a characteristic form on $M'$ such that the respective scalar Levi form $\mathcal{L}_\Theta(\Bar L_1, \Bar L_2) = \frac{1}{2i}\Theta([L_1,\Bar L_2])$ is positive semidefinite. If $S$ is strongly pseudoconvex at $\pp$, then $\mathcal{L}_\Theta|_{T^{0,1}_\pp S}$ is strictly positive definite, hence, by elementary linear algebra, $\mathcal{L}_\Theta$ has at least $N'-1-K$ positive eigenvalues and as the $K$-dimensional leaf $\eta_\pp$ through $\pp$ is a complex manifold and therefore Levi-flat, the other $K$ eigenvalues have to be zero.
Consider now $R_\pp(\Bar L,V)$ for $\Bar L \in \mathcal{V}(M)$ and $V \in T_\pp\eta_\pp$. Decompose $\Bar L|_\pp = U + W$ for $U \in T^{0,1}_\pp S_\pp$ and $W \in T^{0,1}_\pp \eta_\pp$. As proven in Lemma \ref{lem:tensor}, $R_\pp(W,V) = 0$, hence $R_\pp(\Bar L,V) = 0$ iff $R_\pp(U,V) = R^S_\pp(U,V) = 0$. This implies that $\ker R_\pp(\cdot,V) = \ker R^S_\pp(\cdot, V) \oplus T^{0,1}_\pp \eta_\pp$, proving the second claim.
\end{proof}

\begin{figure}[h]
    \centering
    \includegraphics[width=0.4\textwidth]{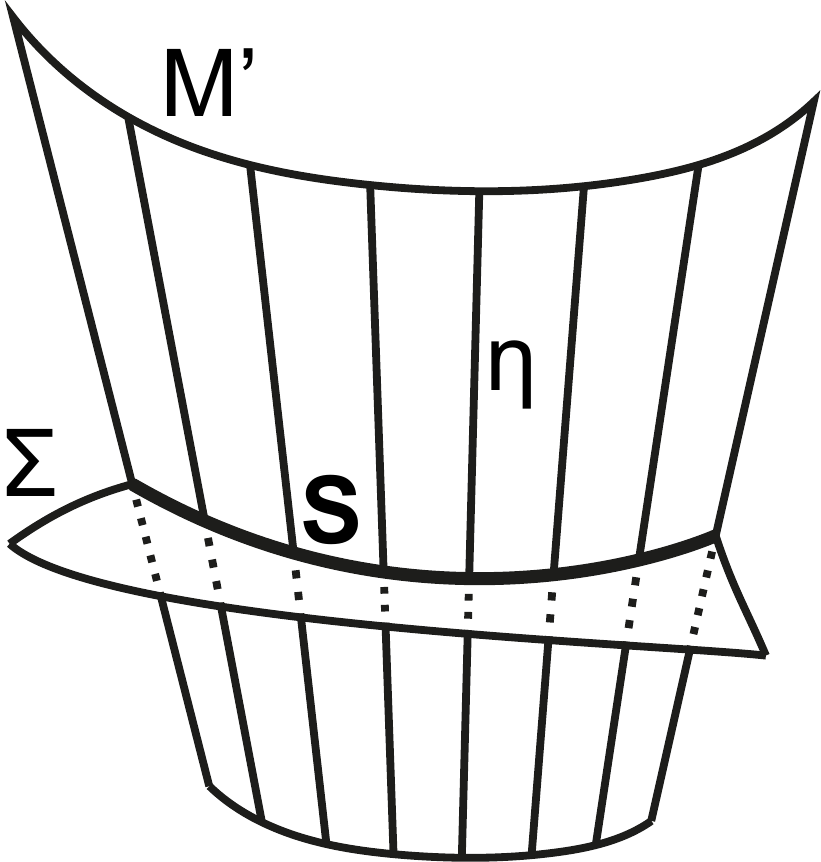}
    \caption{The setup from \Cref{lem:setup}.}
    \label{fig:my_label}
\end{figure}

\subsection{The case $\nu = 0$} If the kernel of $R$ is of minimal dimension even at a single point, Proposition \ref{prop:pointwise} implies \Cref{cor:always}, fully generalizing the result on the tube over the light cone; before we 
give the proof, we first show that the tube over the light cone satisfies
$\nu=0$.

\begin{exa}
\label{exa:lightcone}
Let $M' \subseteq \C^{N'}$ be the tube over the light cone. It is foliated by complex lines, at any point $p' \in M'$ with $\Re(\pp) \neq 0$ the Levi form of $M'$ has exactly one zero eigenvalue, and $\nu_\pp = 0$.
\end{exa}

\begin{proof}
Recall that the tube over the light cone is defined as the set of points $z \in \mathbb{C}^{N'}$ such that $\Re(z_1)^2 + \dots \Re(z_{N'-1})^2 = \Re(z_{N'})^2$. It is a smooth real hypersurface where $\Re(z_{N'}) \neq 0$, and foliated by complex lines $q + t\left( \Re(q_1),\dots,\Re(q_{N'-1}),\Re(q_{N'})\right)$, $q \in M'$. Indeed, let us check that
\begin{align*}
    &\Re\left(q_1 + t\Re(q_1)\right)^2 + \dots + \Re\left(q_{N'-1} + t\Re(q_{N'-1})\right)^2 \\ &= 
    (1+\Re(t))^2 \left( \Re(q_1)^2 +\dots + \Re(q_{N'-1})^2 \right) \\ &=
    (1+\Re(t))^2 \Re(q_{N'})^2 = \Re\left(q_{N'} + t\Re(q_{N'})\right)^2.
\end{align*}
The hypersurface $M'$ is pseudoconvex, since the tube over the interior of the light cone is convex.
The hypersurface $\Sigma = \{z \in \C^{N'}: z_{N'} = p'_{N'}\}$ through $\pp \in M'$ is transversal to $\eta_\pp$ and intersects $M'$ in $S=\{z \in \C^{N'}: z_{N'} = p'_{N'}, \Re(z_1)^2 + \dots + \Re(z_{N'-1})^2 = \Re(p'_{N'})^2\}$, which is a strongly pseudoconvex CR submanifold of $\Sigma$ because it is a tube over a strongly convex real manifold.
To obtain the setup of Lemma \ref{lem:setup}, it now suffices to calculate $R^S_\pp(\Bar L|_\pp,V|_\pp)$ for a single section $V$ of $T\eta$ (since $T\eta$ is one-dimensional).
Take $V(q) = \left( \Re(q_1),\dots,\Re(q_{N'-1}),\Re(q_{N'})\right)$ for $q \in S$, and consider a CR vector $\Bar L|_\pp \in T^{0,1}_\pp S$. Since $\Bar L|_\pp \Re(q_{N'}) = 0$, $L|_\pp V \in T_\pp\Sigma$, and because $T_\pp\Sigma$ and $T_\pp\eta$ lie in general position, $R_{p'}^S(\Bar L|_\pp,V|_\pp) = 0$ if and only if $\Bar L|_\pp V = 0$, i.e. $\Bar L|_\pp (\Re(q_j)) = 0$ for $j = 1,\dots,N'$. But since $\Bar L_\pp q_j = 0$, this is the case if and only if $\Bar L_\pp \Bar q_j = 0$ as well, hence $\Bar L|_\pp \in T_\pp^{0,1}S \cap T_\pp^{1,0}S = \{0\}$, which proves that $\nu_\pp = 0$.
\end{proof}


\begin{proof}[Proof of \Cref{cor:always}]
If $h$ is not generically smooth, there exists an open set $U_0 \subseteq M$ where $h$ is nowhere smooth. \Cref{lem:r0always} always yields $r_0 \geq 1$, thus we may apply \Cref{thm:main}, with $k = 0$ and $l = 1$, to obtain a point $p \in M$, a neighborhood $U \subseteq U_0$ of $p$ and a continuously differentiable CR family of formal complex curves $(\Gamma_\xi)_{\xi \in U}$ such that $\Gamma_\xi$ is tangential to $M'$ to infinite order at $h(\xi)$. For any $q \in U$ we have $\nu_{h(q)} = 0$, i.e. $\ker R(\cdot,\gamma_t(q)) = T^{0,1}_{h(q)}\eta$, hence by \Cref{prop:pointwise} we find $h_\ast T^{0,1}_q M \subseteq T^{0,1}_{h(q)} \eta$.

Near points $q \in U$ where $h$ is regular enough, this means that $h^{-1}(\eta_{h(q)})$ integrates the complex tangent bundle and thus, by minimality of $M$, has to contain an open neighborhood of $q$. First, take a small open set $O' \subseteq h(U)$ where coordinates adapted to the foliation may be chosen, and hence the restricted foliation $\eta|_{O'}$ is equipped with a manifold structure. Let $\pi: O' \rightarrow \eta|_{O'}$ denote the projection onto the foliation, given by $\pi(q) = \eta_q\cap O'$. Since the rank of a continuously differentiable map is lower semicontinuous, we can find an open subset $\tilde U \subseteq h^{-1}(O')$ where $\pi \circ h$ is of constant rank. By the rank theorem, $h^{-1} \circ \pi^{-1}(\eta_{h(q)}) = h^{-1}(\eta_{h(q)}) =: E_q$ is a submanifold of $\tilde U$, and $V \in T_q M$ satisfies $(\pi\circ h)_\ast V = 0$ if and only if $V \in T_q E_q$. But as $h_\ast T_q^{0,1}M \subseteq T^{0,1}_{h(q)}\eta$, we infer that $T_q^{0,1}M \subseteq \C T_q E_q$ for any $q \in \tilde U$, which by minimality implies that $E_q$ is an open neighborhood of $q$ in $\tilde U$ already.

We have thus shown that an nonempty open subset of $U$ is mapped into a single leaf $\eta_{h(q)}$ of the foliation of $M'$. It is left to prove that actually, all of $M$ will be mapped into $\eta_{h(q)}$. Let $\mathcal{K}$ denote the closure of the set of all points $q \in M$ which possess an open neighborhood that is mapped entirely into $\eta_{h(q)}$. We will show that $\mathcal{K}$ is open. For $p \in \mathcal{K}$, take a neighborhood $O' \subseteq \C^{N'}$ of $h(p)$ where \emph{the connected component} of $\eta_{h(q)} \cap O'$ containing $h(p)$ is given as the vanishing set of a holomorphic map $\pi: O'\rightarrow \C^{K}$. Then $\pi \circ h: h^{-1}(O') \rightarrow \C^K$ is a CR map, and since $p \in \mathcal{K}$, $\pi \circ h$ must vanish on open sets arbitrarily close to $p$. But as any CR function on a connected minimal submanifold which vanishes on an open subset already vanishes identically (a consequence e.g. of \cite[Theorem III.3.13]{berhanu-cordaro-hounie}), all of $h^{-1}(O')$ must be mapped entirely into $\eta_{h(q)}$ and thus $p$ lies in the interior of $\mathcal{K}$. Now, $\mathcal{K}$ is both open and closed, and thus connectedness of $M$ implies $\mathcal{K} = M$.
\end{proof}

\section{CR transversal maps and
proof of \Cref{cor:transversal}}
If we want to treat positive $\nu_\pp$, we have to assume more about the map and the source manifold; for example, if $M'=M\times \C$ for a 
strictly pseudoconvex $M\subset \C^N$, then $\nu = N-1$ everywhere, and our usual example 
$z \mapsto (z,\varphi(z))$ for some finitely smooth but nonsmooth CR function 
$\varphi$ will yield nonsmooth CR maps. 

The approach we take here is based on the fact that if  $h_\ast T_p^{0,1}M$ and $T_{p'}\eta_{p'}$ intersect trivially,  then we can allow $\nu_\pp$ to be greater than zero provided that $h_\ast T_p^{0,1}M$ has enough dimensions. In particular, this 
occurs if $M$ is a uniformly pseudoconvex hypersurface with sufficiently many positive Levi eigenvalues, and $h$ satisfies a commonly considered nondegeneracy condition, that of \emph{CR-transversality}.

\begin{defi}[CR-transversality]
A CR map $h:M \rightarrow M'$ between hypersurfaces $M$ and $M'$ is called CR-transversal at $p\in M$ if
$T_{h(p)}^{0,1}M' + T_{h(p)}^{1,0}M' + h_\ast \C T_pM = \C T_{h(p)}M'$.
\end{defi}

If $M$ is actually strongly pseudoconvex and $h:M\rightarrow M'$ is CR-transversal, then $h_\ast T_p^{0,1}M$ has maximal dimensions and intersects $T^{0,1}_\pp \eta_\pp$ trivially. This is very well known and a key component in the proof of many regularity results (see e.g. \cite{regularitymx}. We summarize for later use the following statement:

\begin{lem}
\label{lem:crtransversal}
Consider a pseudoconvex hypersurface $M'$, a \emph{strongly pseudoconvex} hypersurface $M$ and a $C^2$-smooth \emph{CR-transversal} CR map $h: M \rightarrow M'$ mapping $p\in M$ to $p' \in M'$. Then $h$ is an immersion, $\dim h_\ast T_p^{0,1}M = \dim_{CR} M$, and $h_\ast T_p^{0,1}M \cap \mathcal{N}_\pp = \{0\}$, where $\mathcal{N}_\pp \subseteq T_\pp^{0,1}M'$ denotes the null space of the Levi form of $M'$ at $\pp$.
\end{lem}


With this fact in mind, it is clear that strict pseudoconvexity of $M$ and CR-transver\-sality of $h$ come together to imply that there are many CR directions available along $h(M)$ where some obstructions to the existence of CR families of infinitely tangential formal complex curves, encoded in $R$, might exist. If $M$ is just pseudoconvex, these CR directions may be obtained by considering a strongly pseudoconvex slice of $M$ of maximal dimension; this 
observation is the basis of 
\Cref{cor:transversal}. 


\begin{proof}[Proof of \Cref{cor:transversal}]
Take an open neighborhood $U \subseteq M'$ of $\pp$ where $\nu_{q'} < n_+$ for all $q' \in U$. Let $\Sigma \subseteq \C^N$ be a complex manifold of dimension $1+ n_+$ such that $T_p\Sigma + T^c_p M = T_p \C^N$, and such that the Levi form of $M$ is positive definite when restricted to $T_p^{0,1}M \cap T_p^{0,1}\Sigma$. Since both transversality and positive definiteness are open conditions, after possibly shrinking $U$, the intersection $S := \Sigma \cap M \cap h^{-1}(U)$ is a strongly pseudoconvex CR submanifold of $M$ with $n_+ = \dim_{CR} S$. Note that $T_qS$ for $ q \in S$ always contains a transversal tangent vector, hence the restricted map $h|_S$ is CR transversal if and only if $h$ is.

By \Cref{lem:crtransversal}, $h|_S$ is an immersion, $\dim h_\ast T_q^{0,1} S = n_+$ and $h_\ast T_q^{0,1}S \cap T_{h(q)}^{0,1}\eta = \{0\}$ for any $q \in S$, where $\eta$ denotes the Levi foliation of $M'$, and thus $T^{0,1}\eta$ is the Levi null space. This is precisely the situation of \Cref{cor:r1}, hence $r_1 \geq 1 + \dim_{CR}S = 1 + n_+$, after possibly restricting $U$ again.

Assume now that $h$ was nowhere smooth on an open set $\tilde O \subset h^{-1}(U)$. Then \Cref{thm:main}, with $k=1$, $l=1 + n_+$, would yield $q \in \tilde O$ (mapped to $q' \in M'$) and a continuously differentiable CR family of formal complex curves $(\Gamma_\xi)_{\xi \in \tilde O_1}$ defined on a neighborhood $\tilde O_1 \subseteq \tilde O$ of $q$. But now \Cref{prop:pointwise} implies $h_\ast T_{q}^{0,1} M \subseteq \ker R_{q'}(\cdot, \gamma_t(q))$, thus we find that
\begin{align*}
    K + \nu_{q'} &\geq \dim_\C \ker R_{q'}(\cdot, \gamma_t(q)) \geq \dim_\C\left( T_{q'}^{0,1}\eta + h_\ast T_q^{0,1} M\right) \\
    &\geq \dim_\C\left(  T_{q'}^{0,1}\eta  + h_\ast T_q^{0,1} S \right) = K + n_+,
\end{align*}
which contradicts the assumption $\nu_{q'} < n_+$. Therefore $h$ must be $\cinfty$-smooth on a dense open subset of $O := h^{-1}(U)$.
\end{proof}

\subsection{Connection to $2$-nondegeneracy}
\label{sec:2ndeg}
In this section, we discuss the relationship of the invariant $\nu$ to finite nondegeneracy. 

We first recall that a  CR submanifold $M \in \C^N$ is called $k$-nondegenerate, for $k \in \mathbb{N}$, at a point $p \in M$ if
\begin{align*}
    T_p^{0,1}\C^N = \left<\{\bar L_1 \bar L_j \rho_w: \rho \in \mathscr{I}_M(p), \bar L_1,\dots,\bar L_j \in \mathcal{V}_p(M), 0\leq j \leq k \in \}\right>;
\end{align*}
equivalently, its identity map 
satisfies $r_k (p)=N$, where the
$r_k$ are defined in \cref{ss:formalfol}.

It turns out that $1$-nondegeneracy is equivalent to Levi-nondegeneracy, by essentially the same argument as in the proof of \Cref{lem:r1} in \cref{sec:surfaces}.

For a pseudoconvex hypersurface which is not strictly pseudoconvex, the next step  is $2$-nondegeneracy. 

\begin{prop}
Let $M\subset \C^N$ be a uniformly Levi-degenerate hypersurface with $N-K-1$ nonzero Levi eigenvalues and Levi foliation $\eta$, and for $p \in M$, consider $\nu_p$ as in \Cref{def:nu}. Then the following are equivalent:
\begin{itemize}
    \item $M$ is $2$-nondegenerate at $p$,
    \item There exists \emph{no} germ of a section $V \in \Gamma_p(T^{1,0}\eta)$ such that $[\bar L_1, [\bar L_2, V]]\rvert_p \in T_p^{1,0}M \oplus T_p^{0,1}M$ for all $\bar L_1, \bar L_2 \in \Gamma_p(T^{0,1}M)$,
    \item $\nu_p < N-K-1$, i.e. $\nu_p$ is not maximal.
\end{itemize}
\end{prop}

\begin{proof}
First, we note that by the product rule and the fact that two smooth defining functions of hypersurfaces differ by a smooth nonzero factor, it is sufficient to consider a single defining function $\rho$.
To work in a covariant setting, we consider $T'M = \left< d z_1|_M, \dots, d z_N|_M\right> \subset \C T^\ast M$, the space of $(1,0)$-forms on $M$.

Let the Lie derivative of a $(1,0)$-form $\omega$ with respect to a CR vector field $\bar L$ be denoted by $\mathcal{T}_{\bar L} \omega$. By suitably extending $\bar L$ and $\omega$ to a neighborhood of $p$ in $\C^N$, we compute 
\begin{align*}
    \left(\mathcal{T}_{\bar L} \omega\right)_j &= \mathcal{T}_{\bar L} \omega (\frac{\partial}{\partial z_j}) = 
    d\omega(\bar L, \frac{\partial}{\partial z_j}) + \frac{\partial}{\partial z_j} \omega(\bar L) = d\omega(\bar L, \frac{\partial}{\partial z_j}) \\
    &= \bar L \omega(\frac{\partial}{\partial z_j}) - \frac{\partial}{\partial z_j} \omega(\bar L) - \omega([\bar L, \frac{\partial}{\partial z_j}]) = \bar L \omega_j,
\end{align*}
using Cartan's magic formula and the fact that $\omega$ annihilates the $0,1$-vector fields $\bar L$ and $[\bar L, \frac{\partial}{\partial z_j}]$. Hence, in this setting, taking Lie derivatives means just taking derivatives component-wise.

For a defining function $\rho$, consider now $i\partial \rho = i\sum_j \rho_{z_j} dz_j$. This form differs from the real contact form $\Theta = i(\partial \rho - \bar \partial \rho)$ only by a multiple of $d\rho$, which vanishes along $M$. By the previous calculation, $M$ is $2$-nondegenerate if and only if $\partial \rho$, $\mathcal{T}_{\bar L_1}\rho$ and $\mathcal{T}_{\bar L_2}\mathcal{T}_{\bar L_1}\rho$ together span all of $T_p'M$, where $\bar L_1$ and $\bar L_2$ range across all germs of CR vector fields at $p$.

The hypersurface $M$ is $2$-\emph{degenerate} at $p$ if and only if there exists a nonzero vector $V \in \C T_p M$ such that $\partial \rho(V) = 0$, $(\mathcal{T}_{\bar L_1} \partial \rho)(V) = 0$ and $(\mathcal{T}_{\bar L_2}\mathcal{T}_{\bar L_1} \partial \rho)(V) = 0$. The first condition ensures that $V \in T_p^{1,0} M + T_p^{0,1} M$, and since $T'M$ annihilates $T^{0,1}M$, we may assume without loss of generality that $V \in T_p^{1,0}M$. Next, we calculate
\begin{align*}
    0 &= (\mathcal{T}_{\bar L_1} \partial \rho)(V) = \tfrac{1}{i}(\mathcal{T}_{\bar L_1} \Theta)(V) = \tfrac{1}{i} d\Theta(\bar L_1, V) = 2 \mathcal{L}_\Theta(\bar L_1, \bar V),
\end{align*}
for all $\bar L_1 \in T^{0,1}M$, hence $V \in T^{1,0}\eta$. We extend $V$ to a local section of $T^{1,0}\eta$. Then, the third condition yields
\begin{align*}
    0 &= i(\mathcal{T}_{\bar L_2}\mathcal{T}_{\bar L_1} \partial \rho)(V) 
    = \left(\mathcal{T}_{\bar L_2} d\Theta(\bar L_1,\cdot) \right) =
    d\left( d\Theta(\bar L_1,\cdot)\right)(\bar L_2,V) \\
    &= \bar L_2 d\Theta(\bar L_1, V) - V d\Theta(\bar L_1, \bar L_2) - d\Theta(\bar L_1, [\bar L_2, V]) \\
    &= - \bar L_1 \Theta([\bar L_2, V]) + [\bar L_2,V] \Theta(\bar L_1) + \Theta([\bar L_1, [\bar L_2,V]]) = \Theta([\bar L_1, [\bar L_2,V]]).
\end{align*}
Thus, $M$ is $2$-degenerate at $p$ if and only if there exists a section $V \in \Gamma_p(T^{1,0}\eta)$ such that $[\bar L_1, [\bar L_2,V]]\rvert_p \in T_p^{1,0}M \oplus T_p^{0,1}M$, for all $\bar L_1, \bar L_2 \in \Gamma_p(T^{0,1}M)$.

Writing $V = \sum_j V^j \frac{\partial}{\partial z_j}$ and $\bar L_2 = \sum_j \bar L_2^j \frac{\partial}{\partial \bar z_j}$, we find that $[\bar L_2, V] = \sum_j (\bar L_2 V^j) \frac{\partial}{\partial z_j} - \sum_j (V \bar L_2^j) \frac{\partial}{\partial \bar z_j}$. As $[\bar L_1,[\bar L_2,V]] \in T_p^{1,0}M \oplus T_p^{0,1}M$ for any $\bar L_1 \in \mathcal{V}_p$, the $(1,0)$-part of $[\bar L_2,V]$ must lie in the Levi null space of $M$, i.e. $\sum_j (\bar L_2 V^j) \frac{\partial}{\partial z_j}|_p \in T_p^{1,0}\eta$. Almost tautologically, this is the case if and only if $V_r := (V^1,\dots,V^j) \in T\eta \subseteq T\C^N$ satisfies $\bar L_2 V_r \in T\eta$ for all CR vector fields $\bar L_2$, i.e. if and only if $\nu_p = N-K-1$.
\end{proof}

The calculation of $\nu$ for boundaries of classical symmetric domains in \cref{sec:csd} thus also provides a way of concluding that these hypersurfaces are $2$-nondegenerate.
In \cite{regularitymx}, Xiao proves that a merely $\diffable 2$-smooth CR map from a strongly pseudoconvex hypersurface with $n_+$ positive Levi eigenvalues into a $2$-nondegenerate uniformly pseudoconvex hypersurface of precisely $n_+$ positive Levi eigenvalues must be $\cinfty$-smooth \emph{everywhere}. The point is that, under these conditions on the Levi eigenvalues, the image of the source manifold already contains all relevant vector fields to conclude that the map itself is $2$-nondegenerate in the sense of Lamel \cite{La1}, and thus as regular as the source and target manifolds themselves. 

The invariant $\nu_p$ contains more subtle information than mere $2$-nondegeneracy of the target manifold, as is evinced by the sharp bounds on the number of Levi eigenvalues of the source manifold achieved in \Cref{thm:zusammenfassung}. One could hope that the condition $\nu_{p'} < n_+$ from \Cref{cor:transversal} already suffices to conclude that the CR map $h$ itself is $2$-nondegenerate (at least on a dense open subset), but this very likely true only in Xiao's special case. It is not at all easy to find interesting examples of this behaviour, as everywhere finitely 
nondegenerate, pseudoconvex hypersurfaces of Levi number exceeding $2$ are
extremely scarce and notoriously hard to construct; we refer the reader to the discussion in Baouendi, Ebenfelt, and Zaitsev's paper \cite{MR2581969}.



\section{Applications to holomorphic maps}
\label{sec:domains}
In this section, we give the proof of \Cref{cor:holomorphic} in which we apply our results on CR transversal CR maps between smooth hypersurfaces in complex Euclidean space to holomorphic maps which extend to CR maps between the smooth part of their source and target domains' boundaries. Before presenting the proof, we need to collect some preliminary results. It is a well known fact that such holomorphic maps give rise to CR transversal boundary maps, as long as the target domain satisfies a suitable convexity condition (cf. \cite{fornaessthesis} for strongly pseudoconvex or \cite{regularitymx} for convex target domains). Indeed, mere pseudoconvexity of the target suffices to guarantee CR transversality of the boundary map.

\begin{prop}
\label{prop:boundarytransversality}
Let $\Omega \subseteq \C^N$, $\Omega' \subseteq \C^{N'}$ be domains and let $M \subseteq \partial \Omega$, $M' \subseteq \partial \Omega'$ be smooth real hypersurfaces contained in the smooth parts of $\partial \Omega$ and $\partial \Omega'$, respectively. Suppose that $\Omega'$ is pseudoconvex at $M'$.

Then any holomorphic map $F: \Omega \rightarrow \Omega'$ which extends to a map of regularity $C^{1,\epsilon}$ on $\Omega \cup M$ and maps $M$ into $M'$ is CR transversal along $M$.
\end{prop}

This proposition as well as its proof parallel Proposition 9.10.5. in \cite{baouendi}, but as the latter result is only stated for equidimensional mappings with smooth boundary extension, a proof for \Cref{prop:boundarytransversality} shall nevertheless be presented.
The proof hinges on the following observation by Diederich \& Fornaess \cite[p. 133, Remark b]{dfexponent}.

\begin{lem}
\label{lem:localdfexponent}
Let $\Omega \subseteq \C^N$ be a pseudoconvex domain and $p \in \partial \Omega$ be a point in the smooth part of its boundary. Take any $\eta \in (0,1)$. Then there exists a neighborhood $U$ of $p$ and a defining function $\rho$ of $\Omega$ on $U$ such that $-(-\rho)^\eta$ is strictly plurisubharmonic on $\Omega \cap U$.
\end{lem}

As paper \cite{dfexponent} is mainly interested in global properties of pseudoconvex domains, the proof of \Cref{lem:localdfexponent} is merely hinted at in a remark. For a full proof, see \cite[Thm. 2.2.17]{baouendi}.

\begin{proof}[Proof of \Cref{prop:boundarytransversality}]
Suppose that $F:\Omega \rightarrow \Omega'$ extends to a CR map that is not CR-transversal at a point $p \in M$. Choose $0<\delta<\epsilon$, and let $\eta = \frac{1+\delta}{1+\epsilon}$. By \Cref{lem:localdfexponent}, there exists a neighborhood $U$ of $F(p)$ and a defining function $\rho \in C^\infty(U,\R)$ for $M'$ such that $U \cap \Omega' = \rho^{-1}(-\infty,0)$ and such that $-(-\rho)^\eta$ is strictly plurisubharmonic on $U \cap \Omega'$. By assumption, the normal derivative of $\rho \circ F$ vanishes at $p$, hence $\rho \circ F$ has a critical point at $p$. By Hölder continuity of the derivative, $\rho \circ F (z) = \mathcal{O}(|z-p|^{1+\epsilon})$ near $p$. But this implies $-(-\rho)^\eta \circ F(z) = \mathcal{O}(|z-p|^{1+\delta})$ on $F^{-1}(U\cap \Omega')$, and thus the normal derivative of $-(-\rho)^\eta \circ F$ at $p$ vanishes as well. Since $-(-\rho)^\eta \circ F$ as a pull-back of a subharmonic function along a holomorphic map is subharmonic as well, and since $-(-\rho)^\eta \circ F$ clearly has a local maximum at $p$, the normal derivative of $\rho^\eta \circ F$ at $p$ is nonzero by the Hopf lemma, a contradiction.
\end{proof}

A holomorphic map inherits the regularity of its induced boundary map, immediately allowing the transferral of \Cref{cor:transversal} to holomorphic maps.

\begin{proof}[Proof of \Cref{cor:holomorphic}]
As $\Omega'$ is pseudoconvex at $M'$, \Cref{prop:boundarytransversality} yields CR transversality of the boundary map $h := H|_{M'}$. Thus, the hypothesis of \Cref{cor:transversal} is met, and $h$ is $\cinfty$-smooth on a dense open subset $O \subseteq M$ of a neighborhood of $p$ in $M$. By Theorem 7.5.1. in \cite{baouendi}, $H$ then extends to a $\cinfty$-smooth map on $O\cup \Omega$.
\end{proof}

\begin{rem}
In particular, if $H: \Omega \rightarrow \Omega'$ is a \emph{proper} holomorphic map and extends to a $\diffable{N'-n_+}$-smooth map on $\Bar\Omega$, it maps $M$ into the topological boundary of $\Omega'$. If a point $p\in M$ is known to be mapped to some $p' \in M'$, an open neighborhood $U \subseteq M$ of $p$ is then also mapped into $M'$, and the hypothesis of \Cref{cor:holomorphic} is satisfied.
\end{rem}

\section{Maps into boundaries of classical symmetric domains}
\label{sec:csd}
Before we discuss the CR geometry of the boundaries of the classical symmetric domains that we need to apply our results, let us recall some basic facts. We call a bounded domain $\Omega \subset \C^N$ a \emph{bounded symmetric domain} if it exhibits a biholomorphic involution $h_p: \Omega \rightarrow \Omega$ for every point $p \in \Omega$ which has $p$ as an \emph{isolated} fixed point and which satisfies $Dh(p) = -\mathbb{I}_N$ (cf. \cite{symmetricdomain}).

A bounded domain $\Omega$ may be equipped with the \emph{Bergman metric}, a Hermitian metric with the property that each biholomorphism on $\Omega$ is an isometry. Considered together with this metric, a bounded symmetric domain $\Omega$ becomes a special case of a \emph{Hermitian symmetric space}, i.e. a manifold equipped with a Hermitian metric such that each point is an isolated fixed point of some involutive isometry. It can be shown that the group of isometries of such manifolds acts transitively, therefore they can be expressed as the coset space of the the stabilizer group of $\Omega$, defined as the group of isometries leaving a chosen point fixed, in the full isometry group of $\Omega$ (cf. \cite{symmetricspace}). This allows the classification of bounded symmetric domains by Lie group techniques.

According to \cite{symmetricdomain}, any bounded symmetric domain is biholomorphic to a direct product of \emph{irreducible} boun\-ded symmetric domains. Irreducible bounded symmetric domains fall into four series of classical symmetric domains as well as two exceptional cases (as classified by Cartan, cf. \cite{cartan}). The study of regularity of proper holomorphic maps into classical symmetric domains, and consequently of CR maps into their boundaries, has been taken up by Xiao in \cite{regularitymx}; for important
applications of maps between classical symmetric domains, we refer the reader to e.g. Kim and Zaitsev's paper on rigidity of these maps \cite{MR3343893}. We will adopt Xiao's naming convention for the classical symmetric domains, which differs from Cartan's original numbering only in swapping domains of the third and fourth kind.

Finally,  let us briefly recall the singular value decomposition from linear algebra. A matrix $A \in \C^{m\times n}$, $m \leq n$ may always be decomposed as $A = U \Sigma V^\ast$, where
\begin{enumerate}
    \setlength\itemsep{0cm}
    \item $U \in \C^{m\times m}$ is a unitary matrix, forming a basis of eigenvectors for $AA^\ast$,
    \item $\Sigma \in \C^{m\times n}$ is a diagonal matrix with nonnegative entries, and
    \item $V \in \C^{n\times n}$ is another unitary matrix, forming a basis of eigenvectors for $A^\ast A$.
\end{enumerate}
The diagonal entries of $\Sigma$, $0\leq \sigma_1 \leq \dots \leq \sigma_m$ are called the \emph{singular values} of $A$. They are given by the square roots of the eigenvalues of the (Hermitian, positive semidefinite) matrix $AA^\ast$, or equivalently by the square roots of the $m$ largest eigenvalues of $A^\ast A$. The largest singular value of $A$ yields the operator norm of $A$ with respect to the standard scalar product on $\C^m$ and $\C^n$. The matrix $V$ of \emph{right singular vectors} may be freely chosen among the orthonormal eigenvector bases of $A^\ast A$, which then fixes $\Sigma U = AV$, and therefore those columns of $U$ corresponding to nonzero singular values, the \emph{left singular vectors}.

\subsubsection{Classical domains of the first kind}

We will denote the examples in the first series by $D_I^{m,n}$ for $1 \leq m \leq n$. According to Cartan \cite{cartan}, they may be realized as
\begin{align*}
    D_I^{m,n} = \{Z \in \C^{m\times n}: \mathbb{I}_m-Z Z^\ast \text{ is strictly positive definite.}\}
\end{align*}

The condition $\mathbb{I}_m-Z Z^\ast > 0$ is equivalent to the largest singular value of $Z$ being strictly bounded by one, i.e. $\lVert Z \rVert < 1$, where $\lVert \cdot \rVert$ always denotes
 the usual Euclidean matrix norm (or vector norm, respectively). The boundary of $D_I^{m,n}$ is thus given by the set of matrices of norm $1$, equivalently, by those matrices which have $1$ as their largest singular value. This set is a smooth manifold where exactly one singular value is $1$. 
To see this, consider the characteristic polynomial $P(\lambda) = \det(\lambda\mathbb{I}_m - ZZ^\ast)$ of $ZZ^\ast$, which has a simple zero at $1$ by assumption. Now $\rho(Z) := \det(\mathbb{I}_m - ZZ^\ast)$ has nonvanishing gradient, since \[\rho(Z + \mu Z) = \det(\mathbb{I}_m - |1 + \mu|^2 ZZ^\ast) = |1 + \mu|^{2m}P(|1+\mu|^{-2})\] has nonvanishing derivative, providing us with a defining equation.

Let us denote this smooth piece of the boundary by $M_I^{m,n}$. Because $M_I^{m,n}$ bounds the convex region $D_I^{m,n}$, it is a pseudoconvex real hypersurface. The singular value decomposition will translate to a foliation of $M_I^{m,n}$ by complex (in fact, complex linear) manifolds, setting $M_I^{m,n}$ up as an interesting example case for applying \Cref{cor:transversal}. The following result should be compared to Proposition 1.2 in \cite{regularitymx}, where only strongly pseudoconvex hypersurfaces in $\C^{m{+}n{-}1}$ are considered.

\begin{prop}
\label{prop:d1}
Let $m \geq n \geq 2$ and $M$ be a pseudoconvex hypersurface with at least $n_+ \geq m{+}n{-}3$ positive Levi eigenvalues.
Then every CR-transversal CR map $h$ of regularity $C^{mn-n_+}$ from $M$ into $\MOne$ is generically smooth.
\end{prop}

In the course of our proof, we will be utilizing  the boundary orbit theorem, which states that the Lie group of biholomorphic automorphisms of $D_I^{m,n}$ also acts transitively on $\MOne$ by ambient biholomorphisms (see e.g. \cite{wolf} or \cite[proof of Lemma 2.2.3]{MN}). This allows us to analyze $\MOne$ around points which are particularly easy to understand from the matrix model alone, namely the rank one matrices  in $\MOne$.

Indeed, suppose $h: M \rightarrow \MOne$ is nowhere smooth on a neighborhood $O\subset M$ of a point $p \in M$. Any matrix $ab^\ast$ for vectors $a\in \C^m$, $b \in \C^n$ of unit norm is contained in $\MOne$, since it has a lone singular value $1$. By the boundary orbit theorem, there exists a biholomorphic map $F_{h(p)}$ defined on a neighborhood of $h(p)$ mapping $h(p)$ to $ab^\ast \in \MOne$ and $\MOne$ into itself. Then $\tilde h := F_{h(p)} \circ h$ is a CR-transversal CR map taking $p$ to $ab^\ast$, which is nowhere smooth on $O$ as well. At $ab^\ast$, we check directly that the prerequisites to apply \Cref{cor:transversal} are fulfilled.

\begin{lem}
\label{d1structure}
Let $a\in \C^n, b \in \C^m$ be unit vectors. Around $ab^\ast$, the pseudoconvex hypersurface $\MOne$ is foliated by $(n{-}1){\times}(m{-}1)$-dimensional complex (linear) manifolds. Its Levi form has exactly $m+n-2$ positive eigenvalues, and $\nu_{ab^\ast} = m+n-4$.
\end{lem}

If $\tilde h$ was nowhere smooth around $p$, this would contradict \Cref{cor:transversal}, as $\nu_{ab^\ast} = m+n-4 < n_+$. This proves \Cref{prop:d1}.

\begin{proof}[Proof of \Cref{d1structure}]
Let $\Sigma \subset \C^{m\times n}$ be the set of $m \times n$ matrices of rank (exactly) $1$, which is an $(m{+}n{-}1)$-dimensional holomorphic manifold containing $ab^\ast$. In linear coordinates such that $a = (1,0,\dots,0)^T \in \C^n$ and $b = (1,0,\dots,0)^T \in \C^m$, $\Sigma$ is parametrized holomorphically by 
\[(z_1,\dots,z_m,  w_2,\dots,w_n) \mapsto (z_1,\dots,z_m)^T(1,w_2\dots,w_n)\] around $ab^\ast = (1,0,\dots,0)^T(1,0,\dots,0)$. To see explicitly that this map is one-to-one near $ab^\ast$, for a matrix $Z \in \Sigma$, let $w$ be the (unique) intersection of $(\ker Z)^\bot$ and $b + \left<b\right>^\bot$. Then $w^\ast = (1,w_2,\dots,w_n)$ and $Z(w)/\lVert w \rVert^2 = (z_1,\dots,z_m)^T$.

The hypersurface $S: = \Sigma \cap \MOne$ of rank one matrices with norm $1$ is strongly pseudoconvex. Indeed, because $\lVert uv^\ast \rVert = \lVert u \rVert\lVert v \rVert$, a defining equation for $S$ is given by 
\begin{align*}
    \rho(z_1,\dots,z_m,w_2,\dots,w_n) = (|z_1|^2 + \dots + |z_m|^2)(1 + |w_2|^2 + \dots + |w_n|^2) - 1 = 0,
\end{align*}
with (real) Hessian $2 \mathbb{I}_{2(m+n-1)}$ at $ab^\ast$, implying that $S$ is actually strongly convex.

The singular value decomposition expresses any matrix $A \in \MOne$ as $uv^\ast + B$, where $u$ and $v$ are unit singular vectors (unique up to simultaneous multiplication by $\lambda \in \mathbb{S}^1$) corresponding to the lone singular value $1$, and the uniquely determined matrix $B \in \C^{m \times n}$ satisfies $B v = 0$, $u^\ast B = 0$ and $\lVert B \rVert < 1$. Conversely, every matrix $uv^\ast + B$ of this type lies in $\MOne$. The set of all $B \in \C^{m\times n}$ with $u^\ast B = 0$ and $Bv = 0$ is an $(m-1)\times(n-1)$-dimensional vector space, and thus the affine planes 
\begin{align*}
    \eta_{uv^\ast} := \{uv^\ast + B: B \in \C^{m\times n}, u^\ast B = 0, Bv = 0\}
\end{align*} for $uv^\ast \in S$ provide the desired foliation $\eta$ of $\MOne$ near $ab^\ast$. The tangent bundle $T\eta$ at $S$ is just given by $T_{uv^\ast}\eta = \{B \in \C^{m\times n}: Bv = 0, u^\ast B = 0\}$.

Having established the setup from \Cref{lem:setup}, all that remains is to compute the tensor $R^S$ at $ab^\ast$. Take $B_0 \in T_{ab^\ast} \eta$, $B_0 \neq 0$. If we define $B(Z)$ for $Z \in \C^{m\times n}$ by
\begin{align*}
    B(Z) = (\mathbb{I}_m - ZZ^\ast)B_0(\mathbb{I}_n - Z^\ast Z),
\end{align*}
then $B(uv^\ast)$ provides a section of $T\eta$ along $S$ satisfying $B(ab^\ast) = B_0$, since
\begin{align*}
    u^\ast B(uv^\ast) &= u^\ast(\mathbb{I}_m - u u^\ast)B_0(\mathbb{I}_n - v v^\ast) = 0, \\
    B(uv^\ast) v &= (\mathbb{I}_m - u u^\ast)B_0(\mathbb{I}_n - v v^\ast)v = 0 \textrm{ and} \\
    B(ab^\ast) &= (\mathbb{I}_m - a a^\ast)B_0(\mathbb{I}_n - b b^\ast) = \mathbb{I}_mB_0\mathbb{I}_n = B_0.
\end{align*}

Returning to $S$, we work out that $\mathfrak{V} := \{a\beta^\ast + \alpha b^\ast: \alpha \in \left< a \right>^\bot \subset \C^m, \beta \in \left< b \right>^\bot \subset \C^n\}$ is the complex tangent space of $S$ at $ab^\ast$. To show that $\mathfrak{V}$ is tangential, we take two curves $\gamma_1: (-\varepsilon,\varepsilon) \rightarrow \mathbb{S}^{2m-1}$ and $\gamma_2: (-\varepsilon,\varepsilon) \rightarrow \mathbb{S}^{2n-1}$ through $a$ and $b$, respectively, satisfying $\dot\gamma_1(0) = \alpha$ and $\dot\gamma_2(0) = \beta$. Now $\gamma_1\gamma_2^\ast$ defines a curve in $S$, and $\frac{d}{dt}|_{t=0} \left(\gamma_1(t)\gamma_2(t)^\ast\right) = a\beta^\ast + \alpha b^\ast$. The space $\mathfrak{V}$ is parametrized in a complex linear way by $(\alpha, \Bar \beta) \mapsto a \Bar \beta^T + \alpha b^\ast$, where $(\alpha,\Bar \beta)$ lies in the $(m{+}n{-}2)$-dimensional complex subspace of $\C^{m+n}$ defined by $a^\ast \alpha = 0$, $\Bar \beta^T b = \beta^\ast b = 0$. To check that this map is indeed injective, test $\alpha b^\ast + a \beta^\ast$ from right and left with $b$ and $a^\ast$, respectively, to obtain $\alpha$ and $\Bar\beta$ again. Since the complex tangent space of $S$ has only $\dim_{\C} \Sigma - 1 = m+n-2$ dimensions, $T_{ab^\ast}^c S = \mathfrak{V}$ follows.

Consider a CR vector $\Bar L|_{ab^\ast} = \frac{1}{2}(X + iJX)$ for $X \in T^c_{ab^\ast}S$ and write $X = a\beta^\ast + \alpha b^\ast$. Then the holomorphic curve $\gamma(t) = (a+t\alpha)(b+\Bar t\beta)^\ast$ is tangential to $\Bar L|_{ab^\ast}$ at $t=0$. Observing that both $\lVert a + t \alpha \rVert^2= 1 + |t|^2 \lVert \alpha \rVert^2$ and $\lVert b + \Bar t \beta \rVert^2$ are constant to first order, we obtain
\begin{align*}
    \Bar L|_{ab^\ast} B &= \frac{d}{d\Bar t}\Big|_{t=0} B \circ \gamma(t) =
    \frac{d}{d\Bar t}\Big|_{t=0} \left(\mathbb{I}_m - \gamma(t)\gamma(t)^\ast \right)B_0\left(\mathbb{I}_n - \gamma(t)^\ast\gamma(t) \right)\\
    & = \frac{d}{d\Bar t}\Big|_{t=0} \left(\mathbb{I}_m - \lVert b + \Bar t \beta \rVert^2 (a+t\alpha)(a+t\alpha)^\ast\right)B_0\left(\mathbb{I}_n - \lVert a + t \alpha \rVert^2 (b+\Bar t\beta) (b+\Bar t\beta)^\ast\right) \\ &= -a\alpha^\ast B_0 - B_0 \beta b^\ast.
\end{align*}

Recall that the scalar product in $\C^{m\times n}$ may be written as $(A|B) = \mathrm{tr}(A^\ast B)$. By commuting matrices inside the trace we see that for any $Z$ in $T_{ab^\ast}\eta$, 
\begin{align*}
    \mathrm{tr}\left((\Bar L|_{ab^\ast} B)^\ast Z\right) &= -\mathrm{tr}(B_0^\ast\alpha a^\ast Z + b\beta^\ast B_0^\ast Z) 
    = -\mathrm{tr}(B_0^\ast\alpha a^\ast Z) - \mathrm{tr}(B_0^\ast Z b \beta^\ast ) = 0,
\end{align*}
since $a^\ast Z = Z b = 0$. This means that $R^S_{ab^\ast}(\Bar L|_{ab^\ast}, B_0) = -a\alpha^\ast B_0 - B_0 \beta b^\ast$, because the projection onto $T^\bot\eta$ is already taken care of, and $\Bar L|_{ab^\ast} \in \ker R^S_{ab^\ast}(\cdot,B_0)$ if and only if $a\alpha^\ast B_0 + B_0 \beta b^\ast = 0$. Testing this with $a^\ast$ and $b$ from left and right, respectively, we obtain $\alpha \in \ker B_0^\ast$ and $\beta \in \ker B_0$. Since $B_0^\ast a = 0$ and $B_0 b = 0$ already, both kernels have codimension at least one in $\left<a\right>^\bot$ and $\left<b\right>^\bot$, respectively, thus $\dim_{\C}\ker R^S_{ab^\ast}(\cdot,B_0) \leq m+n-4$, implying $\nu_{ab^\ast} = m+n-4$.
\end{proof}

\Cref{prop:d1} gives all dimensions where a statement this simple is meaningful and possible. If $M$ has more than $m+n-2$ positive Levi eigenvalues, there is no CR-transversal map from $M$ to $\MOne$, since, by \Cref{lem:crtransversal}, the target manifold would need to have at least as many positive Levi eigenvalues as the source. If $M$ has less than $m+n-3$ positive Levi eigenvalues, there are nowhere smooth CR-transversal CR maps into $\MOne$ of arbitrarily high regularity.

\begin{exa}
Let $\hat S$ be the strongly pseudoconvex hypersurface given by the $(m-1)\times(n-1)$ matrices of rank one and norm $1$. Then $\hat S$ has $m+n-4$ positive Levi eigenvalues. Take a $C^k$, but nowhere $\cinfty$-smooth CR function $\phi$ on $\hat S$ with $|\phi|<1$. Then $h(Z) = \begin{pmatrix} Z & 0 \\ 0 & \phi \end{pmatrix}$ gives a nowhere smooth CR-transversal CR map $h: \hat S \rightarrow \MOne$ of regularity $C^k$.
\end{exa}

\begin{proof}
Regularity is obvious from the component-wise definition. CR-transversality always holds for the graph map of a CR function, i.e. $h: M \rightarrow M \times \C$, $h(p) = (p,\phi(p))$, since $T^c (M\times \C) \cong T^c M \times \C$, $T(M \times \C) \cong TM \times \C$ and $\mathbb{P}_{TM} \circ h_\ast \cong \mathrm{id}$ together imply that any transversal vector $v \in TM \setminus T^cM$ maps into a transversal vector again. That $h(Z) \in \MOne$ follows from the singular value computations in the proof of \Cref{d1structure}.
\end{proof}

\subsubsection{Classical domains of the second kind}
These classical symmetric domains, denoted by $\DTwo$, $m \geq 2$, are given as the sets of \emph{skew symmetric} complex $m{\times}m$ matrices with norm less than $1$. Equivalently,
\begin{align*}
\DTwo = \left\{ Z \in \C^{m\times m}: Z^T = - Z, \mathbb{I}_m - Z^\ast Z > 0 \right\}.
\end{align*}

Every nonzero singular value of a skew symmetric matrix $Z$ occurs with even multiplicity. Suppose $u$ is a right singular vector corresponding to a singular value $\sigma$, which is equivalent to $Z^\ast Z u = \sigma^2 u$. Then $v := \sigma^{-1}\overline{Zu}$ is another right singular vector corresponding to $\sigma$, since it follows from $Z^\ast = \Bar Z^T = - \Bar Z$ that $Z^\ast Z v = - \sigma^{-1} \Bar Z Z \Bar Z \Bar u = - \sigma^{-1} \Bar Z \overline{\Bar Z Z u} = \sigma^{-1} \Bar Z \overline{Z^\ast Z u} = \sigma \Bar Z \Bar u = \sigma^2 v$, and $v^\ast v = \sigma^{-2}u^T Z^T \Bar Z \Bar u = \sigma^{-2}\overline{u^\ast Z^\ast Z  u} = 1$. Furthermore, $v$ and $u$ are orthogonal, and $u = -\sigma^{-1}\overline{Zv}$:
\begin{align*}
    \sigma u^\ast v &= u^\ast \Bar Z \Bar u = (u^\ast \Bar Z \Bar u)^T = u^\ast \Bar Z^T \Bar u = - u^\ast \Bar Z \Bar u \Rightarrow u^\ast v = 0, \\
    -\sigma^{-1} \overline{Z v} &= -\sigma^{-2} \overline{Z \Bar Z \Bar u} = - \sigma^{-2} \Bar Z Z u = \sigma^{-2} Z^\ast Z u = u.
\end{align*}

The boundary of $\DTwo$ is given by those skew symmetric matrices with norm $1$. It is a smooth manifold where exactly the largest two singular values are $1$. We will denote this smooth piece of the boundary by $\MTwo$. Let us postpone checking that $\MTwo$ is a manifold to the proof of \Cref{d2structure}.

\begin{prop}
\label{prop:d2}
Let $m \geq 4$ and $M$ be a pseudoconvex hypersurface with at least $n_+ \geq 2m-7$ positive Levi eigenvalues.
Then any CR-transversal CR map $h$ of regularity $C^{\frac{m(m-1)}{2}-n_+}$ from $M$ into $\MTwo$ is generically smooth.
\end{prop}

Completely analogously to the situation of \Cref{prop:d1}, this follows from the boundary orbit theorem for $\DTwo$, which allows us to map each point in $\MTwo$ to $\pp := ab^T - ba^T$ for orthonormal $a,b \in \C^m$ by an automorphism of $\MTwo$, and from the following structural properties.

\begin{lem}
\label{d2structure}
Let $a, b \in \C^m$ be orthonormal vectors. Around $p' := ab^T - ba^T \in \MTwo$, the pseudoconvex hypersurface $\MTwo$ is foliated by $\frac{(m-2)(m-3)}{2}$-dimensional complex (linear) manifolds. Its Levi form has exactly $2m-4$ positive eigenvalues, and $\nu_{p'} = 2m-8$.
\end{lem}

\begin{proof}
As the intersection of the linear subspace of skew symmetric matrices with the convex matrix norm unit ball, $\DTwo$ is convex and $\MTwo$ is a pseudoconvex hypersurface.

The set $\Sigma$ of skew symmetric matrices of rank two is a $(2m{-}3)$-dimensional complex manifold around $p'$. In coordinates where $a = (1,0,\dots,0)^T$ and $b=(0,1,0,\dots,0)^T$, it is parametrized around $p'$ by 
\begin{align*}
    (z_3,...,z_m,w_2,...,w_m) \mapsto (1,0,z_3,...,z_m)^T(0,w_2,...,w_m) - (0,w_2,...,w_m)^T(1,0,z_3,...,z_m).
\end{align*}
To check surjectivity, let $\Bar u$ and $\Bar v$ be two right singular vectors corresponding to the only nonzero singular value $\sigma$, chosen such that $Z\Bar u = -\sigma v$ and $Z \Bar v = \sigma u$. Then $Z = u(\sigma v)^T - (\sigma v)u^T$. Since $a^\ast(ab^T-ba^T)\Bar b = 1$, $a^\ast Z \Bar b \neq 0$ near $p'$, implying that at least one of $a^\ast u$ or $a^\ast v$ is nonzero. By substituting $(-v,u)$ for $(u,v)$ if necessary, we can arrange $a^\ast u \neq 0$. Let $\tilde u = u$, $\tilde v = \sigma (v - \frac{a^\ast v}{a^\ast u} u)$, then $a^\ast \tilde v = 0$ and $Z = \tilde u \tilde v^T - \tilde v \tilde u^T$. Note that $a^\ast Z \Bar b \neq 0$ now implies $b^\ast \tilde v \neq 0$. Let $z = \frac{1}{a^\ast \tilde u} \tilde u - \frac{b^\ast \tilde u}{(a^\ast \tilde u)(b^\ast \tilde v)} \tilde v$ and $w = (a^\ast \tilde u) \tilde v$. Then we have $a^\ast z = 1$, $b^\ast z = 0$, $a^\ast w = 0$ and $Z = zw^T - wz^T$, proving that $Z$ is in the range of our parametrization. To check that it is an immersion, it suffices to calculate $\frac{\partial}{\partial z_j}(zw^T-wz^T) = e_je_2^T - e_2e_j^T$, $3 \leq j \leq m$ and $\frac{\partial}{\partial w_k}(zw^T-wz^T) = e_1e_k^T - e_ke_1^T$, $2 \leq k \leq m$, since these are evidently $\C$-linearly independent matrices.

The set $S = \Sigma \cap \MTwo$ of skew symmetric rank two matrices with norm $1$ is a strictly pseudoconvex hypersurface in $\Sigma$. To show this, first note that for orthogonal vectors $\alpha, \beta \in \C^m$, we have 
\begin{align*}
    \lVert \alpha \beta^T - \beta \alpha^T \rVert^2 &= \lVert \left( \alpha \beta^T - \beta \alpha^T\right)^\ast\left(\alpha \beta^T - \beta \alpha^T \right)\rVert = \lVert \lVert \alpha\rVert^2 \Bar \beta\beta^T + \lVert \beta \rVert^2 \Bar \alpha\alpha ^T\rVert \\ 
    &= \lVert \alpha \rVert^2 \lVert \beta \rVert^2 \lVert \mathrm{diag}(1,1,0,\dots,0)\rVert = \lVert \alpha \rVert^2 \lVert \beta \rVert^2.
\end{align*}

The standard Euclidean scalar product on $\C^{m\times m}$ coincides with the Frobenius scalar product $(A|B) = \mathrm{tr}(A^\ast B)$. For a matrix $Z = \alpha \beta^T - \beta \alpha^T$ with orthogonal $\alpha, \beta \in \C^m$, the Frobenius norm works out to 
\begin{align*}
    \sqrt{\mathrm{tr}(Z^\ast Z)} = \sqrt{\mathrm{tr}(\lVert \alpha \rVert^2 \Bar \beta \beta^T + \lVert \beta \rVert^2 \Bar \alpha \alpha^T)} = \sqrt{2}\lVert \alpha \rVert \lVert \beta \rVert.
\end{align*}
Therefore, the Frobenius norm and the matrix norm agree up to a constant on $\Sigma$, and $S = \Sigma \cap \MTwo = \Sigma \cap \sqrt{2}\mathbb{S}^{2m^2-1}$ is strongly pseudoconvex, as it is given by the intersection of a complex manifold with a strongly convex hypersurface.

The singular value decomposition expresses $Z \in \MTwo$ as $uv^T - vu^T + B$, where $\Bar u$ and $\Bar v$ are right singular vectors corresponding to the double singular value $1$ satisfying $Z \Bar u = -v$ and $Z \Bar v = u$, and $B$ satisfies $B \Bar u = B \Bar v = 0$, $u^\ast B = v^\ast B = 0$ and $\lVert B \rVert < 1$. By linearity, we have $B^T = -B$, implying that $B \Bar u = B \Bar v = 0$ and $u^\ast B = v^\ast B = 0$ are equivalent. In coordinates where $u=(1,0,0,\dots,0)$ and $v = (0,1,0,\dots,0)$, the conditions $B=-B^T$ and $B \Bar u = B \Bar v = 0$ simply mean that $B$ is a skew symmetric matrix with the first two rows and columns empty. We conclude that the affine planes $\eta_{uv^T-vu^T} = uv^T-vu^T + \{B \in \C^{m\times m}: B=-B^T, B\Bar u = B\Bar v = 0\}$ for $uv^T-vu^T \in S$ provide a foliation of $\MTwo$ by $\frac{(m-2)(m-3)}{2}$-dimensional complex manifolds, and that $\MTwo$, as an embedded piece of a vector bundle over $S$, is indeed a manifold.

The complex tangent space $T_\pp^cS$ at $\pp = ab^T - ba^T$ will be given by the complex vector space $\mathfrak{V} := \{a\beta^T - \beta a^T + \alpha b^T - b\alpha^T: \alpha,\beta \in \left< a,b\right>^\bot \subset \C^m \}$. To show tangency, consider the complex curve $\gamma(t) = (a+t\alpha)(b+t\beta)^T - (b+t\beta)(a+t\alpha)^T$, with tangent vector $\gamma_t(0) = a\beta^T - \beta a^T + \alpha b^T - b\alpha^T$. It is contained in $\Sigma$ and tangential to $\MTwo$, the latter because $\lVert \gamma(t) \rVert^2 = \lVert a + t\alpha \rVert^2\lVert b + t\beta \rVert^2 - \left| (a+t\alpha)^\ast(b+t\beta)\right|^2 = \lVert a \rVert^2 \lVert b \rVert^2 + \mathcal{O}(|t|^2)$, hence $\gamma_t(0) \in T_\pp^cS$. Since $\mathfrak{V}$ is isomorphic to $\left< a,b\right>^\bot$ by the map $\gamma_t(0) \mapsto \left( \gamma_t(0)\Bar b, -\gamma_t(0) \bar a\right)$, it has $2m-4 = \dim_{CR}S$ dimensions, and $T_\pp^cS = \mathfrak{V}$.

Given $B_0 \in T_\pp\eta$, the map $B(Z) = (\mathbb{I}_m-ZZ^\ast)B_0(\mathbb{I}_m-Z^\ast Z)$ again provides a section of $T\eta$ along $S$, since for orthonormal $u,v \in \C^m$,
\begin{align*}
    B(uv^T-vu^T) &= (\mathbb{I}_m-uu^\ast-vv^\ast)B_0(\mathbb{I}_m-\Bar uu^T-\Bar vv^T) = -B(uv^T-vu^T)^T \\
    B(uv^T-vu^T)\Bar u &= (\mathbb{I}_m-uu^\ast-vv^\ast)B_0(\Bar u - \Bar u) = 0, \text{ and } B(uv^T-vu^T)\Bar v = 0.
\end{align*}
Taking a CR vector $\Bar L|_\pp \in T_\pp^{0,1} S$ with real part $\frac{1}{2}\left(a\beta^T - \beta a^T + \alpha b^T - b\alpha^T\right)$ and the curve $\gamma(t) = (a+t\alpha)(b+t\beta)^T-(b+t\beta)(a+t\alpha)^T$, we first obtain
\begin{align*}
    \gamma(t)\gamma(t)^\ast &= \lVert a+t\alpha \rVert^2(b+t\beta)(b+t\beta)^\ast + \lVert b+t\beta \rVert^2(a+t\alpha)(a+t\alpha)^\ast \\
    &- t\Bar t(\beta^T\Bar\alpha)(a+t\alpha)(b+t\beta)^\ast - t \Bar t (\alpha^T \Bar\beta)(b+t\beta)(a+t\alpha)^\ast \\
    &= (b+t\beta)(b+t\beta)^\ast + (a+t\alpha)(a+t\alpha)^\ast + \mathcal{O}(|t|^2), \\
    \gamma(t)^\ast\gamma(t) &= \overline{(b+t\beta)}(b+t\beta)^T + \overline{(a+t\alpha)}(a+t\alpha)^T + \mathcal{O}(|t|^2),
\end{align*}
which simplifies the calculations for $R_\pp^S$ significantly. We obtain
\begin{align*}
    \Bar L|_\pp B &= \frac{d}{d\Bar t}\Big|_{t=0} B \circ \gamma(t) = \frac{d}{d\Bar t}\Big|_{t=0} \left(\mathbb{I}_m - \gamma(t)\gamma(t)^\ast \right)B_0\left(\mathbb{I}_m - \gamma(t)^\ast\gamma(t) \right) \\
    &= \frac{d}{d\Bar t}\Big|_{t=0} \Big(\big(\mathbb{I}_m - (b+t\beta)(b+t\beta)^\ast - (a+t\alpha)(a+t\alpha)^\ast\big)B_0 \\
    &\cdot \big(\mathbb{I}_m - \overline{(b+t\beta)}(b+t\beta)^T - \overline{(a+t\alpha)}(a+t\alpha)^T)\big) + \mathcal{O}(|t|^2)\Big) \\
    &= - b \beta^\ast B_0 - a \alpha^\ast B_0 - B_0 \Bar \beta b^T - B_0 \Bar \alpha a^T.
\end{align*}
By the same calculations as in the proof of \Cref{d1structure}, we find that this already gives $R_\pp^S(\Bar L|_\pp,B_0) = - b \beta^\ast B_0 - a \alpha^\ast B_0 - B_0 \Bar \beta b^T - B_0 \Bar \alpha a^T$, and that $\Bar L|_\pp \in \ker R_\pp^S(\cdot,B_0)$ if and only if $\alpha, \beta \in \ker \Bar B_0$. As a nonzero skew symmetric matrix, $\Bar B_0$ has at least two nonzero singular values, hence $\textrm{codim}_\C \ker\Bar B_0 \geq 2$. Since $\Bar B_0 a = \Bar B_0 b = 0$, and $\alpha, \beta \in \left< a, b\right>^\bot$, we obtain $\textrm{codim}_\C\ker R^S_\pp(\cdot,B_0) \geq 4$ and thus $\nu_\pp = 2m-8$.
\end{proof}

As in \Cref{prop:d1}, there are counterexamples to regularity if $M$ has exactly $2m-8$ positive Levi eigenvalues.

\begin{exa}
Let $\hat S \subset M_{II}^{m-2}$ be the strongly pseudoconvex hypersurface of antisymmetric $(m-2)\times(m-2)$ matrices of rank two and norm $1$. It has $2m-8$ positive Levi eigenvalues. Given a $C^k$-smooth, but nowhere $\cinfty$-smooth CR function $\phi$ on $\hat S$ strictly bounded by $1$, the map $h:\hat S \rightarrow \MTwo$ given by
\begin{align*}
    h(Z) = \begin{pmatrix}
        Z & 0 & 0 \\
        0 & 0 & -\phi \\
        0 & \phi & 0
        \end{pmatrix}
\end{align*}
is a $C^k$-smooth, but nowhere $\cinfty$-smooth CR-transversal CR function.
\end{exa}

\subsubsection{Classical domains of the third kind}
Domains of the third kind $\DThree$ are given by the sets of \emph{symmetric} complex $m{\times}m$ matrices with norm less than $1$. Equivalently,
\begin{align*}
    \DThree = \left\{ Z \in \C^{m\times m}: Z^T = Z, \mathbb{I}_m - Z^\ast Z > 0 \right\}.
\end{align*}
Here the regularity result obtained from \Cref{cor:transversal} only holds for $M \subset \C^m$.

\begin{prop}
\label{prop:d3}
Let $m \hspace{2pt} {\geq} \hspace{2pt} 2$ and $M$ be a pseudoconvex hypersurface with at at least $n_+ \geq m-1$ positive Levi eigenvalues.
Then every CR-transversal CR map $h$ of regularity $C^{\frac{m(m+1)}{2}-n_+}$ from $M$ into $\MThree$ is generically smooth.
\end{prop}

Let us note in passing that a nontrivial CR transversal CR map from $M$ into $\MThree$ can only exist if the number of positive Levi eigenvalues of $M$ does not exceed $m-1$, thus this result only truly concerns uniformly pseudoconvex hypersurfaces.

\Cref{prop:d3} is a consequence of the boundary orbit theorem for $\DThree$, which tells us that every point $Z \in \MThree$ may be mapped to $aa^T$ for a unit vector $a \in \C^m$ by an ambient biholomorphism mapping $\MThree$ into itself. Almost completely analogously to the case of $\MOne$, the following structural properties hold.

\begin{lem}
\label{d3structure}
Let $a \in \C^m$ be a unit vector. Around $aa^T \in \MThree$, the pseudoconvex hypersurface $\MThree$ is foliated by $\frac{m(m-1)}{2}$-dimensional complex (linear) manifolds. Its Levi form has exactly $m-1$ positive eigenvalues, and $\nu_{aa^T} = m-2$.
\end{lem}

\begin{proof}
As the intersection of the convex set of matrices of norm less than $1$ with the linear subspace of symmetric matrices, $\DThree$ is convex, and thus $\MThree$ is pseudoconvex.

Let $\Sigma$ be the $m$-dimensional complex manifold of symmetric matrices of rank $1$. Near $aa^T$, it is parametrized by $z \mapsto zz^T$ for $z \in \C^m$ with $\Re(a^\ast z) > 0$. To check bijectivity, write $Z = \sigma uv^\ast$ for singular vectors $u,v \in \C^m$ and the nonzero singular value $\sigma$. Since $u$ and $v$ lie in the one-dimensional kernels of $ZZ^\ast - \sigma^2\mathbb{I}_m = Z\Bar Z - \sigma^2\mathbb{I}_m $ and $Z^\ast Z - \sigma^2\mathbb{I}_m = \Bar Z Z - \sigma^2\mathbb{I}_m$, respectively, we infer by Cramer's rule that $\lambda u = \Bar v$ for some $\lambda \in \mathbb{S}^1$. Letting $z := \sigma^{\frac{1}{2}}\lambda^{-\frac{1}{2}} \Bar v = \sigma^{\frac{1}{2}}\lambda^\frac{1}{2} u$, we find that $Z = zz^T$. The only indeterminacy here - the choice of sign for the root $\lambda^\frac{1}{2}$ - is fixed by requiring $\Re(z^\ast a) > 0$.

The real hypersurface $S \subset \Sigma$ of rank one matrices with norm $1$ is strongly pseudoconvex. Indeed, as $\lVert zz^T \rVert = \lVert z \rVert^2$, we have that $z \in \mathbb{S}^{2m-1}$ iff $zz^T \in S$, and the map $z \mapsto zz^T$ provides a holomorphic double cover of $S$ by $\mathbb{S}^{2m-1}$, showing that $S \cong \R P^{2m-1}$.

 The complex affine planes $\eta_{ww^T} := \{ww^T + B: B \Bar w = 0, B^T = B\}$ for $w \in \mathbb{S}^{2m-1}$ provide a foliation of $\MThree$ near $aa^T$. As in the proof of \Cref{d1structure}, the singular value decomposition expresses $Z \in \MThree$ as $uv^\ast + B$, where $u$, $v$ are unit vectors (unique up to simultaneous multiplication by $\lambda \in \mathbb{S}^1$), and $B$ satisfies $B^\ast u = B v = 0$ and $\lVert B \rVert < 1$. Since as before, $u$ and $v$ lie in the one-dimensional kernels of $Z\Bar Z - \mathbb{I}_m$ and $\Bar Z Z - \mathbb{I}_m$, respectively, we may express $Z = ww^T + B$ for $w \in \mathbb{S}^{2m-1}$, implying $B^T = B$ by linearity. The condition $B v = B^\ast u = 0$ simplifies to $B \Bar w = 0$. In coordinates where $w = (1,0,\dots,0)^T \in \C^m$, $B \Bar w = 0$ just means that the first column is empty, a condition that is clearly linearly independent of $B^T = B$. Therefore, the space defined by $B^T = B$, $B \Bar w = 0$ is a complex vector space of $\frac{m(m-1)}{2}$ dimensions for $w$ near $a$.

Given $B_0 \in T_{aa^T} \eta$, we prove that $B(Z) = (\mathbb{I}_m - ZZ^\ast)B_0(\mathbb{I}_m - Z^\ast Z)$ provides a section of $T\eta$ along $S$. For $ww^T \in S$,
\begin{align*}
    B(ww^T)\Bar w &= (\mathbb{I}_m - ww^\ast)B_0(\mathbb{I}_m - \Bar w w^T) \Bar w = (\mathbb{I}_m - ww^\ast)B_0(\Bar w - \Bar w) = 0, \\
    B(ww^T)^T &= (\mathbb{I}_m - (\Bar w w^T)^T)B_0^T(\mathbb{I}_m - (ww^\ast)^T) = B(ww^T) \text{\ and}\\
    B(aa^T) &= (\mathbb{I}_m - aa^\ast)B_0(\mathbb{I}_m - \Bar a a^T) = B_0.
\end{align*}
 Consider a CR vector $\Bar L|_{aa^T} \in T_{aa^T}^{0,1} S$. Complex tangent vectors $\alpha \in T_a^c\mathbb{S}^{2m-1}$ are characterized by $\alpha^\ast a = 0$. Since $z \mapsto z z^T$ is holomorphic and onto, we can just plug a suitable complex tangent $t \mapsto a + t\alpha$ into this map to obtain a curve $\gamma(t) = (a + t\alpha)(a + t\alpha)^T$ such that $\Bar L|_{aa^T} = \gamma_\ast \frac{d}{d\Bar t}|_{t=0}$. Then, after rewriting $\gamma(t) = (a + t\alpha)(\Bar a + \Bar t \Bar\alpha)^\ast$, we obtain by the exact same calculation as in \Cref{d1structure} that $R_{aa^T}^S(\Bar L|_{aa^T},B_0) = -a \alpha^\ast B_0 - B_0 \Bar \alpha a^T$. By multiplying from the right with $\Bar a$, we find that $\Bar L|_{aa^T} \in \ker R_{aa^T}^S(\cdot,B_0)$ if and only if $B_0 \Bar \alpha = 0$. Since $\Bar a \in \ker B_0$, the codimension of the kernel of $B_0$ in $\left< \Bar a \right>^\bot$ equals the codimension of the full kernel of $B_0$, hence $\dim_\C \ker R_{aa^T}^S(\cdot,B_0) = m-1-\textrm{codim}_\C \ker B_0 \leq m-2$, implying $\nu_{aa^T} = m-2$.
\end{proof}

Here a counterexample for regularity of CR-transversal maps from source manifolds with less than $m-1$ positive Levi eigenvalues may be constructed in the exact same fashion as in the case of $\MOne$. Let us instead consider a slightly different example map into $\MThree$. 

\begin{exa}
Let $\phi$ be a nowhere smooth CR function of regularity $C^k$ on $\mathbb{S}^{2m-3}$ strictly bounded by $1$. Then the map
$h: \mathbb{S}^{2m-3} \rightarrow \MThree$ given by 
\begin{align*}
    h(z) = \frac{1}{2}(z_1,\dots,z_{m-1},1)^T(z_1,\dots,z_{m-1},1) + \frac{\phi(z)}{2}(z_1,\dots,z_{m-1},-1)^T(z_1,\dots,z_{m-1},-1)
\end{align*}
is a nowhere smooth CR-transversal CR embedding of regularity $C^k$.
\end{exa}

\begin{proof}
We first consider the map $H:\C^{m-1}_z\times\C_w$ given by
\begin{align*}
    H(z,w) = \frac{1}{2}(z_1,\dots,z_{m-1},1)^T(z_1,\dots,z_{m-1},1) + \frac{w}{2}(z_1,\dots,z_{m-1},-1)^T(z_1,\dots,z_{m-1},-1).
\end{align*}
It is a holomorphic immersion on $\C^{m-1} \times \mathbb{B}^1 \subset \C^{m-1} \times \C$, since
\begin{align*}
    \frac{\partial}{\partial z_j}H(z,w) &= \frac{1+w}{2}\left(e_j\big(z_1,\dots,z_{m-1},\frac{1-w}{1+w}\big) + \big(z_1,\dots,z_{m-1},\frac{1-w}{1+w}\big)e_j^T\right), \\
    \frac{\partial}{\partial w}H(z,w) &= \frac{1}{2}(z_1,\dots,z_{m-1},-1)^T(z_1,\dots,z_{m-1},-1),
\end{align*}
where $e_j$ denotes the $j^{th}$ standard unit vector. The matrices $\frac{\partial}{\partial z_j}H(z,w)$ for $1\leq j \leq m-1$ are linearly independent, since their last columns - given by $\frac{1-w}{2}e_j$ - are. Observing that $\frac{\partial}{\partial z_j}H(z,w)_{m,m} = 0$, but $\frac{\partial}{\partial w}H(z,w)_{m,m} = \frac{1}{2} \neq 0$, we conclude that all partial derivatives of $H$ are linearly independent, hence $H$ is immersive.
From the adapted singular value decomposition used in the proof of \Cref{d3structure}, we see that $H$ maps $\mathbb{S}^{2m-3} \times \mathbb{B}^1$ injectively into $\MThree$. Considering the graph map $\Phi: \mathbb{S}^{2m-3} \rightarrow \mathbb{S}^{2m-3} \times \C$, $\Phi(z) = (z,\phi(z))$, which clearly is a $C^k$, but nowhere $\cinfty$-smooth CR embedding of $\mathbb{S}^{2m-3}$, we may write $h = H \circ \Phi$, showing that $h$ is a $C^k$, but nowhere $\cinfty$-smooth CR immersion of $\mathbb{S}^{2m-3}$ into $\MThree$. Note that it is CR-transversal, since $H$ is transversal to $\MThree$, and $\Phi$ was CR-transversal. Since $h$ is an injective immersion of the compact sphere, it is an embedding.
\end{proof}

\subsubsection{Classical domains of the fourth kind}

Somewhat different from the first three series of classical symmetric domains, the models for these domains, denoted by $\DFour$ for $m\geq 2$, are defined by simple quartic inequalities, first given in \cite{cartan}.
\begin{align*}
    \DFour = \left\{ z \in \C^m: z^\ast z < 1, 1 + |z^T z|^2 - 2z^\ast z > 0 \right\}.
\end{align*}
The binding inequality is the second one. Indeed, a point $z \in \partial \DFour$ satisfying $z^\ast z = 1$ also satisfies $|z^Tz|\leq 1$ by Cauchy's inequality, thus $1 + |z^Tz|^2 - 2 z^\ast z \leq 0$. A low-dimensional toy image to have in mind is that of a lens-shaped region defined by $y^2-\frac{1}{4}\left( 1 - x^2 \right)^2<0$, where we discard the unbounded region by requiring $x^2+y^2 < 1$. The smooth part of the boundary of $\DFour$, which we will denote by $\MFour$, is given by those $z\in \C^m$ satisfying $1 + |z^T z|^2 - 2z^\ast z = 0$ and $z^\ast z < 1$.

In fact, $\DFour$ is biholomorphic to the tube domain over the light cone from \Cref{exa:lightcone}. The tube domain over the future light cone is given by $\{(z_1,\dots,z_{m-1},z_{m}) \in \C^m: \Re(z_1)^2 + \dots + \Re (z_{m-1})^2 < \Re(z_{m})^2, \ \Re(z_{m}) > 0\}$. An explicit biholomorphism between the tube domain over the future light cone and $\DFour$ is given in \cite{regularitymx} as
\begin{align*}
    (z_1,\dots,z_{m-1},z_{m}) \mapsto \sqrt{2}i\left( 2\frac{z_1}{\mathcal{F}(z+\mathbf{i})} , \dots , 2\frac{z_{m-1}}{\mathcal{F}(z+\mathbf{i})} , \frac{1+\mathcal{F}(z)}{\mathcal{F}(z+\mathbf{i})}\right),
\end{align*}
where $\mathbf{i}$ denotes the vector $(0,\dots,0,i) \in \C^m$ and where $\mathcal{F}(z) := z_m^2 - z_1^2 - \dots - z_{m-1}^2$ for any $z \in \C^m$. 

Let us nevertheless reprove the regularity result for $\DFour$ by computing the necessary quantities directly from Cartan's representation. As an example point in $\MFour$ to base our calculations on, take $a := (\frac{1}{2},\frac{i}{2},0,\dots,0)^T$. Here, $a^Ta = 0$ and $a^\ast a = \frac{1}{2}$. Contrary to the first three kinds of classical domains, $\MFour$ will necessarily behave exactly like the tube over the light cone.

\begin{prop}
\label{prop:d4}
Let $m \geq 2$ and $M$ be a minimal CR manifold.
Then every CR map $h$ of regularity $C^{m-1}$ from $M$ into $\MFour$ which is of (real) rank $\geq 3$ is $\cinfty$-smooth on a dense open subset of $M$. If $h$ is CR transversal and for some $n_+ \geq 1$, $M$ has at least $n_+$ positive Levi eigenvalues almost everywhere, then initial regularity may be dropped to $\diffable{m-n_+}$.
\end{prop}

This is an immediate consequence of the boundary orbit theorem for $\MFour$, which allows us to take any point in $\MFour$ to $(\frac{1}{2},\frac{i}{2},0,\dots,0)$ by an ambient biholomorphism, and of \Cref{cor:always} (\Cref{cor:transversal} for CR transversal $h$ and $n_+ \geq 1$). The relevant structural properties of $\MFour$ of course do not differ at all from those of the tube over the light cone (\Cref{exa:lightcone}).

\begin{lem}
\label{d4structure}
Let $a \in \C^m$ be such that $a^Ta = 0$ and $a^\ast a = \frac{1}{2}$. Around $a \in \MFour$, the pseudoconvex hypersurface $\MFour$ is foliated by complex lines. Its Levi form has exactly $m-2$ positive eigenvalues, and $\nu_{a} = 0$.
\end{lem}

\begin{proof}
The complex quadric $\Sigma$ defined by $z^Tz = 0$ is a manifold where $z \neq 0$. Its intersection $S$ with $\MFour$ is given by $S = \{w \in \C^m, w^Tw = 0, w^\ast w = \frac{1}{2}\}$. As it is the intersection of a complex manifold with the strongly pseudoconvex sphere given by $w^\ast w = \frac{1}{2}$, it is strongly pseudoconvex itself.

Near a point $w \in S$, the complex line given by $\eta_w(t) = w + t\Bar w$ is contained in $S$. This is proven by straightforward calculation. Since $w^\ast w = \frac{1}{2}$ and $w^T w = 0$, we observe $(w^\ast + \Bar t w^T)(\Bar w + \Bar t w) = \Bar t$ and similar cancellations, and arrive at
\begin{align*}
    & 1 + |\eta_w^T(t) \eta_w(t)|^2 - 2 \eta_w(t)^\ast \eta_w(t) \\ 
    &= 1 + (w + t \Bar w)^\ast \overline{(w + t \Bar w)} (w + t \Bar w)^T (w + t \Bar w) - 2(w + t \Bar w)^\ast(w + t \Bar w) \\
    &= 1 + (w^\ast + \Bar t w^T)(\Bar w + \Bar t w) (w^T + t \Bar w^T) (w + t \Bar w) - 2(w^\ast + \Bar t w^T)(w + t \Bar w) \\
    &= 1 + \Bar t t - (1 + \Bar t t) = 0.
\end{align*}

It remains to calculate the tensor $R^S(\cdot,\Bar w)$ at $a \in S$, since the section $\Bar w$ already spans $T\eta$ along $S$. A vector $v \in T_a^cS$ is characterized by $(a+tv)^T(a+tv) = \mathcal{O}(|t|^2)$ and $(a+tv)^\ast(a+tv) = \frac{1}{2} + \mathcal{O}(|t|^2)$, which is equivalent to $a^Tv = a^\ast v = 0$. Take a CR vector $\Bar L|_a \in T_a^{0,1}S$ with real part $\frac{1}{2}v$, and consider the holomorphic curve $\gamma(t) = a+tv$. Then $\Bar L|_a \Bar w = \frac{d}{d\Bar t}|_{t=0} \overline{(a + t v)} = \Bar v$, and we find that $\Bar v \in T_a^\bot\eta_a = \left< \Bar a \right>^\bot$ already, since $a^\ast v = 0$. Therefore $R^S_a(\Bar L|_a, \Bar a) = \Bar v$ only vanishes if $v$ and thus $\Bar L|_a$ vanish, implying $\nu_a = 0$.
\end{proof}




\bibliography{references}
\bibliographystyle{abbrv}

\end{document}